\documentclass[11pt]{article}

\usepackage{amsmath}
\usepackage{amssymb}
\usepackage{amsthm}
\usepackage{enumerate}

\newcommand{\nor}{\makebox[1em]{$\not$\makebox[.8em]{$\perp$}}\,}

\def \forces{\vdash}

\def \nmodels {\mathop{\not\models}}
\def \nwor {\mathop{\not\perp}^w}
\def \wor  {\perp\!\!\!{^w}}

\def \tp {{\rm tp}}
\def \Sem {{\rm Sem}}

\def \cl {{\rm cl}}

\def \eps {{\mathcal E}}
\def \Aut {{\rm Aut}}
\def \Inv {{\rm Inv}}
\def \Lin {{\rm Lin}}

\def\Mon{{\mathfrak C}}

\def \strok {\!\upharpoonright\!}

 \newtheorem{thm}{Theorem}[section]
 \newtheorem{prop}[thm]{Proposition}
 \newtheorem{lem}[thm]{Lemma}
 \newtheorem{cor}[thm]{Corollary}

  \newtheorem{theo}{Theorem} 

\theoremstyle{definition}
 \newtheorem{exm}[thm]{Example}
 \newtheorem{defn}[thm]{Definition}
 \newtheorem{rmk}[thm]{Remark}
 \newtheorem{assu}[thm]{Assumption}
 \newtheorem{nota}[thm]{Notation}

\title{Asymmetric regular types}
\date{}
\author{Slavko Moconja
\\ Faculty of Mathematics, University of Belgrade, Serbia\and Predrag
Tanovi\'c
\\Mathematical Institute SANU, Belgrade, Serbia}

\usepackage[top=25mm, bottom=3cm, left=25mm, right=25mm]{geometry}

\begin{document}
\maketitle

\begin{abstract}We study  asymmetric regular types. If $\frak p$ is regular and $A$-asymmetric then there exists a strict order such that   Morley sequences in $\frak p$ over $A$ are strictly increasing (we   allow Morley sequences to be indexed by elements of a linear order). We prove that for all $M\supseteq A$  maximal Morley sequences in $\frak p$ over $A$ consisting of elements of $M$ have the same (linear) order type, denoted by $\Inv_{\frak p,A}(M)$, which does not depend on the particular choice of the order witnessing the asymmetric regularity.   In the countable case we determine all possibilities for $\Inv_{\frak p,A}(M)$:   either it can be any countable linear order, or in any $M\supseteq A$ it is a dense linear order (provided that it has at least two elements). Then we study   relationship between $\Inv_{\frak p,A}(M)$ and $\Inv_{\frak q,A}(M)$ when $\frak p$ and $\frak q$ are strongly regular, $A$-asymmetric, and such that   $\frak p_{\strok A}$ and $\frak q_{\strok A}$ are  not weakly orthogonal. We distinguish two kinds on non-orthogonality: bounded and unbounded. In the bounded case we prove that   $\Inv_{\frak p,A}(M)$ and $\Inv_{\frak q,A}(M)$   are  either isomorphic  or anti-isomorphic. In the unbounded case, $\Inv_{\frak p,A}(M)$ and $\Inv_{\frak q,A}(M)$ may have distinct cardinalities but  we prove that their Dedekind completions are either isomorphic or anti-isomorphic. We provide examples of all  four situations.
\end{abstract}

The concept of regularity for global types in an arbitrary first order theory was introduced in \cite{PT}. It was motivated by certain properties of regular (stationary) types in stable theories. The motivation behind the definition is the following: let $\Mon$ be  the monster model of a complete first-order theory $T$ and let  $\frak p$ be a global $A$-invariant type. $\frak p$ induces a division of all definable subsets: those defined by a formula belonging to $\frak p$ are considered to be  `small' subsets, the others are `large'. The division induces a naturally defined operation on the power set of $\Mon$:  
$\cl_\frak p^A(X)$ is the union of all small $AX$-definable subsets. The regularity of $\frak p$ over $A$ means that $\cl_\frak p^B$ is a closure operation on the power set of the locus of $\frak p_{\strok B}$ for all $B\supseteq A$ (this is explained in detail in Section \ref{X_regular}). 
It turns out that there are two essentially distinct kinds of regular types: symmetric and asymmetric. For types of symmetric kind $\cl_\frak p^B$ is   a pregeometry operation   carrying  a naturally defined notion of dimension. In the generically stable case they   share several nice  
properties of regular types in a stable theory; that kind was  investigated in more detail in \cite{Tgs}.  
In the asymmetric case some $\cl_\frak p^B$ is a proper closure operation (the exchange fails). The main consequence  established in \cite{PT} is the existence of a  definable partial order which orders Morley sequences increasingly.  In this paper we study asymmetric regular types in more detail. It  turns out that in most results we do not use the full strength of the original regularity  assumption. What  suffices  is   that the type is weakly regular over $A$ ($\cl_\frak p^A$  is a proper closure operation on the locus of $\frak p_{\strok A}$) and $A$-asymmetric; see Definitions \ref{Dweakly_regular} and \ref{Dassym}.   The main consequence of these assumptions from \cite{PT} is the existence of a definable partial order which orders Morley sequences increasingly. That fact is re-proved in the first part of the following theorem.

\begin{theo}\label{T_regular_implies_totdeg_uvod} Suppose that $\frak p$ is weakly regular and $A$-asymmetric.
 There is an $A$-definable partial order such that any Morley sequence in $p$ over $A$ is strictly increasing.
 For any   $X\subseteq \Mon$ the order type of any maximal Morley sequence (in $p$ over $A$) consisting of elements of $X$   does not  depend   on the particular choice of the sequence and the ordering relation witnessing the asymmetric regularity. 
\end{theo}

Thus instead of having a dimension of a type in a model, we have another invariant: the linear order type of a maximal Morley sequence, denoted by $\Inv_{\frak p,A}(M)$. 
The  partial order $\leq$ mentioned  in the theorem is not uniquely determined, and all of them, when considered on $\frak p_{\strok A}(\Mon)$, have  a strong combinatorial property. Essentially `$\{x,y\}$ cannot be ordered into a Morley sequence in $\frak p$ over $A$' 
is an equivalence relation   on the locus, and its classes are linearly ordered  by  $\leq$  in the same way for any choice of $\leq$.  The equivalence relation  is relatively $\bigvee$-definable over $A$ on the locus of $\frak p_{\strok A}$; the class of $a\models \frak p_{\strok A}$ is denoted by $\eps(a)$. For any $M$ the naturally induced order on all the classes meeting $M$ is also linear and  its order type is  $\Inv_{\frak p,A}(M)$. We investigate possibilities for $\Inv_{\frak p,A}(M)$. 
Two properties of the regular type  are relevant for. The first is  simplicity, i.e the relative definability of `$(x,y)$ is a Morley sequence in $\frak p$ over $A$'  which is essentially equivalent to $\Inv_{\frak p,A}(M)$ being represented by a type-definable object in $M^{eq}$. The second property  is convexity over $A$: whether the order witnessing the $A$-asymmetric regularity can be chosen such that the locus of $\frak p_{\strok A}$ is a convex subset of $\Mon$.  We prove:

\begin{theo}\label{T_omit_uvod}
Suppose that $\frak p$ is weakly regular and $A$-asymmetric. 
\begin{enumerate} [(i)]
\item If $p$ is simple and convex over $A$, and  $\Inv_{\frak p,A}(M)$  has at least two elements, then $\Inv_{\frak p,A}(M)$ is a dense linear order (with or without endpoints).
\item If both $T$ and $A$ are countable and  $p$ is simple  and non-convex over $A$ then  $\Inv_{\frak p,A}(M)$ can be any countable linear order. 
\item If both $T$ and $A$ are countable and $p$ is not simple over $A$  then  $\Inv_{\frak p,A}(M)$ can be any  countable linear order. 
\end{enumerate}
\end{theo} 

Next we study general properties of orthogonality of regular types. We prove:

\begin{theo}\label{T_asym_orth_uvod} 
A regular asymmetric type is orthogonal to any invariant symmetric type. 
$\nor$ is an equivalence relation on the set of asymmetric, regular types. 
\end{theo}

Now, assume that $\frak p,\frak q$ are   regular, $A$-asymmetric types and that the corresponding restrictions are $\nwor$. It is natural to determine the relationship between $\Inv_{\frak p,A}(M)$ and $\Inv_{\frak q,A}(M)$. In general, there may be no connection between them: in  Example \ref{Ex_nonconvex_1} one of them is empty and the other any dense linear order chosen in advance.
 But assuming  in addition that the types are strongly regular, or at least convex in some cases, we get a strong connection.  There are two kinds of $\nwor$ which we call bounded and unbounded. In the bounded case there is an $A$-invariant bijection between  $\eps_{\frak p}$- and $\eps_{\frak q}$-classes, in which case $\Inv_{\frak p,A}(M)$ and $\Inv_{\frak q,A}(M)$ are either isomorphic or anti-isomorphic. In the unbounded case, which may happen only if both types are simple over $A$,  
Dedekind completions of $\Inv_{\frak p,A}(M)$ and $\Inv_{\frak q,A}(M)$ are either isomorphic or anti-isomorphic. Whether an isomorphism or an anti-isomorphism is in question in both cases depends only on  whether $\frak p$ and $\frak q$ commute or not. Since $\Inv_{A,\frak p}(M)$ may be empty, we adopt  the following convention: if $\Inv_{A,\frak p}(M)=\emptyset$ then   its Dedekind completion is a one-element order.

\begin{theo}\label{T_nor_uvod}
Suppose that $\frak p$ and $\frak q$ are regular, convex and $A$-asymmetric, and that $\frak p_{\strok A}\nwor\frak q_{\strok A}$. \begin{enumerate}[(i)]\item  Suppose that $\frak p_{\strok A}\nwor\frak q_{\strok A}$ is of bounded type.  
\begin{enumerate}[1)]
\item If $(\frak p\otimes\frak q)_{\strok A}\neq(\frak q\otimes\frak p)_{\strok A}$ then   $ \Inv_{A,\frak p}(M) $ and $\Inv_{A,\frak q}(M) $ are isomorphic.

\item If $(\frak p\otimes\frak q)_{\strok A}=(\frak q\otimes\frak p)_{\strok A}$ then   $ \Inv_{A,\frak p}(M) $ and $\Inv_{A,\frak q}(M) $ are anti-isomorphic.
\end{enumerate}

\item  Suppose that  $\mathfrak p,\frak q$ are  strongly regular, that $\frak p_{\strok A}\nwor\frak q_{\strok A}$ is of unbounded type. Then both $\frak p$ and $\frak q$ are simple over $A$. We have two cases:  
\begin{enumerate}[1)]
\item If $(\frak p\otimes\frak q)_{\strok A}\neq(\frak q\otimes\frak p)_{\strok A}$ then the Dedekind completions of   $\Inv_{A,\frak p}(M)$ and $ \Inv_{A,\frak q}(M)$ are isomorphic.

\item If $(\frak p\otimes\frak q)_{\strok A}=(\frak q\otimes\frak p)_{\strok A}$ then the Dedekind completions of   $\Inv_{A,\frak p}(M)$ and $ \Inv_{A,\frak q}(M)$ are anti-isomorphic.
\end{enumerate}
\end{enumerate}
\end{theo}

  The paper is organized as follows: In the first part of Section \ref{X_closure} we study proper closure operators, namely those which are not pregeometry operators. Among them we label  totally degenerated ones, which are relevant for studying asymmetric weakly regular types. A totally degenerated closure operator naturally  induces an equivalence relation and a linear order on its classes.
 In the second part of the section we apply these results to $p$-closure operations $\cl^A_p$. 
 In Section \ref{X_regular} we study  general properties of asymmetric, regular and weakly regular types. There an instant application of the results from the second part of Section \ref{X_closure}  proves Theorem   \ref{T_regular_implies_totdeg_uvod}.  
  Section \ref{X_symm_closure} is technical and there we study the equivalence relation induced by $\eps_\frak p$. These results are applied in Section \ref{X_real_invariants}  where we prove Theorem \ref{T_omit_uvod}.  
  In Section \ref{X_orth}     we study general properties of orthogonality of regular types and  prove Theorem \ref{T_asym_orth_uvod}. Section \ref{X_strongreg} is devoted to the proof of Theorem \ref{T_nor_uvod}.

\bigskip
Throughout the paper we use mainly standard notation. Except in the first part of Section \ref{X_closure} where we deal with general closure systems, we will operate in a large saturated structure   $\Mon$.  $a,b,c,..., \bar a,\bar b,...$ denote elements and tuples of elements and $A,B,C,...$ denote small (of cardinality $<|\Mon|$) subsets of the universe. For any formula $\phi(\bar x)$ by $\phi(\Mon)$ we will denote the set of solutions of $\phi(\bar x)$ in 
$\Mon$ and similarly for partial types. $D\subset \Mon^n$ is {\em definable} if it is of the form $\phi(\Mon)$ for some  formula $\phi(\bar x)$. $D$ is {\em type-definable over $A$} if it is an intersection of  sets definable over $A$; 
$D$ is {\em $\bigvee$-definable over $A$} if it is a union of   sets  definable over $A$. If $D$ is type-definable over $A$ and $D'\subseteq D$ then we say that $D'$ is {\em relatively definable within $D$ over $A$} if it is of the form $\phi(\Mon)\cap D$ for some formula $\phi(\bar x)$ over $A$; similarly relative $\bigvee$-definability is defined.

 If $A\subset B$ and $p(\bar x)$ is a type over $B$ then by $p_{\strok A}(\bar x)$ we denote the restriction of $p(\bar x)$ to $A$. {\em Global types} are complete types over the monster and they will be denoted by $\frak p, \frak q,...$. A complete  type $p(\bar x)$ over $B$ {\em does not split over $A\subset B$} if $\varphi(\bar x,\bar b_1)\Leftrightarrow\varphi(\bar x, \bar b_2)\in p(\bar x)$ for all $\bar b_1\equiv\,\bar b_2\,(A)$. 
A global type is {\em $A$-invariant} if it does not split over $A$; a global type is {\em invariant} if it is $A$-invariant for some small set $A$. For an $A$-invariant type $\frak p$   by a {\em linear Morley sequence in $\frak p$ over $A$} we mean a sequence $(\bar a_i\,|\,i\in I)$ where $(I,<)$ is a linear order and $\bar a_i\models \frak p_{\strok A\bar a_{<i}}$ holds for all $i\in I$. Usually, we will omit the word `linear' and the meaning of $A$ will be clear from the context, so we will say simply that it is  a Morley sequence in $\frak p$.  Morley sequences in an $A$-invariant type are indiscernible over $A$.  

  The  {\em product} of invariant types  $\mathfrak
p(\bar x)\otimes\mathfrak q(\bar y)$ is defined by:  
if $\bar a\models \mathfrak p$ and $\bar b\models \mathfrak q_{\strok
\Mon\bar a}$ (in a larger universe) then $\mathfrak p(\bar x)\otimes\mathfrak q(\bar
y)=\tp_{\bar x,\bar y}(\bar a,\bar b/\Mon)$. This definition  reverses the order in the original one from \cite{HP}, the reasons for were explained in \cite{Tgs}.   
The product is associative  but, in general, not commutative. $\mathfrak p$ and $\mathfrak q$ {\em commute} if $\mathfrak p(\bar x)\otimes\mathfrak q(\bar
y)= \mathfrak q(\bar y)\otimes\mathfrak p(\bar
x)$. An invariant type $\frak p$ is {\em symmetric} if $\frak p(\bar x)\otimes\frak p(\bar y)=\frak p(\bar y)\otimes\frak p(\bar x)$; otherwise, it is {\em asymmetric}.  

 Complete types $p,q$ over $A$ are {\em weakly orthogonal}, denoted by $p\wor q$  if $p(\bar x)\cup q(\bar y)$ determines a complete type over $A$. Global types are {\em orthogonal} if they are weakly orthogonal.   Semi-isolation  is defined by: {\em $\bar a$ is semi-isolated by $\bar b$ over $A$}, denoted by $\bar a\in\Sem_A(\bar b)$,  if there is a formula $\varphi(\bar x,\bar y)\in\tp(\bar a,\bar b/A)$ such that $\varphi(\bar x,\bar b)\vdash \tp(\bar a/A)$. Semi-isolation over a fixed set  is transitive: $\bar a\in \Sem_A (\bar b)$ and $\bar b\in \Sem_A(\bar c)$ imply $\bar a\in \Sem_A(\bar c)$.

\section{Closure operations}\label{X_closure}

Suppose that $P$ is a non-empty set  and that $\cl$ is a (unary) operation on the power set of $P$. $\cl$ is a {\em closure operation on $P$} if it satisfies  (for all $X,Y\subseteq P$): 
	\begin{itemize}
	\item $X\subseteq Y$ implies $X\subseteq \cl(X)\subseteq \cl(Y)$;\hfill (Monotonicity)
	\item $\cl(X)=\bigcup\{\cl(X_0)\mid X_0\subseteq X,\,X_0\mbox{ is finite}\}$;\hfill (Finite character)
	\item $\cl(\cl(X))=\cl(X)$;\hfill (Transitivity)
	\end{itemize}
in which case we  say that $(P,\cl)$ is a {\em closure system}. 
A closure operation is a  {\em pregeometry operation} if, in addition, the exchange axiom is satisfied: for every $a,b\in P$, and $X\subseteq P$ $$b\in\cl(X,a)\smallsetminus\cl(X)\mbox{ implies }a\in\cl(X,b),$$ in which case $(P,\cl)$ is a {\em pregeometry}. Closure operations which do not satisfy the exchange  are called {\em proper closure operations}.

\smallskip  If $(P,\cl)$ is a closure system and $P'\subseteq P$ then by $\cl_{P'}(X)=X\cap \cl(P)$ a closure operation on $P'$ is defined, called {\em the restriction of $\cl$ to $P'$}. Usually we will not make distinction in notation between $\cl$ and $\cl_{P'}$; the meaning will be always clear from the context.

\smallskip Well known examples of pregeometries are: $(P,\cl)$ where $\cl$ is the identity; $(V,\cl)$, where $V$ is a vector space  and   $\cl(X)$ is the linear span of $X$; also $(F,\cl)$ is pregeometry, where $F$ is a  field   and   $\cl(X)$ is the set of all algebraic elements over the subfield generated by $X$. In this article we will be interested   in proper closure operations. Any such operation induces a partial order. Let  $a,b,A$ witness the failure of Exchange:  $a\in\cl(Ab)\smallsetminus \cl(A)$  and $b\notin\cl(Aa)$. Then $\cl(Aa)\subsetneq \cl(Ab)$, so the strict inclusion defines a non-trivial strict ordering. This is why   all of the following examples of proper closure systems are based on certain partial orders. 

\begin{exm}\label{Ebasic} Let $(P,\leq)$  be  a partial order and let $(L,<)$ be a linear order; assume that they are non-trivial. The following are examples of proper closure systems.
\begin{enumerate}[(1)]
\item   For $X\subseteq P$ define  $\cl(X)= \{t\in P\mid (\exists\, x\in X)\,t\leq x\}$.  

\item Let $(Q_p,\leq_p)$ be a family of partial orders indexed by $p\in P$. Consider the disjoint union $Q=\dot\bigcup_{p\in P}Q_p$ ordered by: \  $x\leq y$ \ iff \ $x\leq_p y$, for some $p\in P$, or $\pi(x)< \pi(y)$, \ where $\pi$ is  the projection map $\pi: \dot\bigcup_{p\in P}Q_p\longrightarrow P$. 
Let  \ $\cl(X)=\bigcup \{\,\pi^{-1}[\pi(y)]\,|\, (\exists x\in X)\pi(y)\leq \pi(x)\}$. Then $(Q,\cl)$ is a proper closure system.  
    
In this example  $Q_p$'s and $P$ can be recovered from $(Q,\cl)$:  $\cl(x)=\cl(y)$ is an equivalence relation on $Q$; its classes are $Q_p$'s.   $\cl(Q_p)\subseteq\cl(Q_{p'})$   defines a partial order on these classes which is isomorphic to $(P\leq)$.

\item   Of particular interest for us in this paper is a variant of (2) when we take $(L,<)$ instead of $(P,\leq)$ and order the disjoint union of $Q_l$'s for $l\in L$ as in (2). This is an example of a totally degenerated closure system (defined in \ref{Ddegen} below). We will prove in Proposition \ref{P_degenerated} that this example is in a way canonical: any proper, totally degenerated closure system can be obtained in this way. 

\item In a special case of example (3), when all the $Q_l$'s are isomorphic (to $(P,\leq)$ say) we may consider $L\times P$ with the closure defined similarly as in item (2).  

\item Consider $L\times \{0,1\}$ ordered by: $(x,i)\leq (y,j)$ iff $x\leq y$  and $i=j$, and let $\cl$ be defined as in item (1). 
\end{enumerate}
\end{exm}

\begin{defn}\label{Ddegen} Suppose that  $(P,\cl)$ is a closure system. $\cl$  is   {\em totally degenerated} if for all finite, non-empty $X\subseteq P$ there exists $x\in X$ such that $\cl(X)=\cl(x)$.
\end{defn}
 
It is easy to check that the restriction of a totally degenerated closure operation is totally  degenerated; we will use freely this fact throughout the paper without specific mentioning.
All closure operations from Example \ref{Ebasic} are degenerated in the sense that $\cl(X)=\bigcup_{x\in X}\cl(x)$ and   the ones in items (3) and (4) are  totally degenerated. In fact, a totally degenerated closure operation is  degenerated  while the converse is not true: $\cl$ from Example \ref{Ebasic}(5) is degenerated, but not totally degenerated.  

\begin{defn}Let  $(P,\cl)$ be a closure system. For $x,y\in P$ define:
\begin{enumerate}[(i)]
\item $x\leq_\cl y$ \ if and only if \ $\cl(x)\subseteq\cl(y)$;
\item The {\em $\eps$-neighbourhood of $x$} is: \ $\eps(x):=\{y\in P\mid \cl(x)=\cl(y)\}\,.$
\item $P_{\cl}= \{\eps(x)\mid x\in P\smallsetminus\cl(\emptyset)\}$\,.
\item $ \mathbb P_\cl=(P_\cl,\leq_\cl)\ ;$   where $\leq_\cl$ is naturally extended from the order on $P$ to the order  of subsets of $P$: $X\leq_\cl Y$ iff $x\leq_\cl y$ for all $x\in X$ and $y\in Y$.
\end{enumerate}    
\end{defn} 

 According to part (i) of the next proposition $P_{\cl}$ is the quotient of $P\smallsetminus \cl(\emptyset)$ by an equivalence relation defined by: $\eps(x)=\eps(y)$. 
The projection map \ $\pi:P\smallsetminus\cl(\emptyset)\longrightarrow P_\cl$ \ is defined   by $\pi(x)=\eps(x)$. 

\begin{prop}\label{P_degenerated}
Suppose that $\cl$ is a  proper totally degenerated closure operation on $P$. Then:
\begin{enumerate}[(i)]
\item \ $\eps(x)=\eps(y)$ holds if and only if $\cl(x)=\cl(y)$. \     $P_\cl$ is a partition of $P\smallsetminus \cl(\emptyset)$.

\item \  $\eps(x)\mapsto \cl(x)$ defines an isomorphism of linear orders $\mathbb P_\cl$ \  and $(\{\cl(x)\mid x\in P\smallsetminus \cl(\emptyset)\},\subseteq)$\,. 

\item \ $\cl(X)= \cl(\emptyset)\cup\bigcup \{\,\eps(y)\,|\, (\exists x\in X)\,\pi(y)\leq_\cl \pi(x)\}\,.$  
\end{enumerate} 
\end{prop}
\begin{proof}(i) The first claim follows easily from the definition of $\eps$. Suppose that $\eps(x)\cap\eps(y)\neq\emptyset$, and $u\in\eps(x)\cap\eps(y)$. Then, by definition of $\eps$, we have  $\cl(x)=\cl(u)$ and $\cl(y)=\cl(u)$. Hence $ \cl(x) =\cl(y)$ and thus $\eps(x)=\eps(y)$. Therefore, non-disjoint members of $P_\cl$ are equal, so $P_\cl$ is a partition. 

\smallskip 
(ii) First we prove that  $\eps(x)\mapsto \cl(x)$ defines an isomorphism of partial orders. By part (i) it is a bijection.  If $\eps(x)\leq_\cl\eps(y)$, then in particular $x\leq_\cl y$, i.e. $\cl(x)\subseteq \cl(y)$. On the other hand, if $\cl(x)\subseteq\cl(y)$ and $x'\in\eps(x),y'\in\eps(y)$, then $\cl(x')=\cl(x)\subseteq\cl(y)=\cl(y')$, and hence $\eps(x)\leq_\cl\eps(y)$.  Therefore $\eps(x)\leq_\cl\eps(y)$ and $\cl(x)\subseteq\cl(y)$ are equivalent. Thus our map is an isomorphism. Since $\cl$ is totally degenerated, $\cl(x,y)$ is equal to   $\cl(x)$ or $\cl(y)$; thus at least one of $\cl(x)\subseteq \cl(y)$ and $\cl(y)\subseteq \cl(x)$ holds. The orders are linear.

\smallskip
(iii)  Let  $Z=\cl(\emptyset)\cup\bigcup \{\,\eps(y)\,|\, (\exists x\in X)\,\pi(y)\leq_\cl \pi(x)\}$. Assume that $y\in\cl(X)\smallsetminus \cl(\emptyset)= \bigcup_{x\in X}\cl(x) \smallsetminus\cl(\emptyset)$. Then $y\in\cl(x)$ for some $x\in X$, so $\cl(y)\subseteq \cl(x)$. By part (ii) $\eps(y)\leq_\cl\eps(x)$ holds so  $\pi(y)\leq_\cl\pi(x)$. Therefore $y\in Z$ and $\cl(X)\subseteq Z$. The other inclusion is proved similarly: Assume $z\in Z\smallsetminus\cl(\emptyset)$ and let $x\in X$ be such that $\pi(z)\leq_\cl\pi(x)$.  By part  (ii) $\eps(z)\leq_\cl\eps(x)$ implies $\cl(z)\subseteq\cl(x)$, so   $z\in\cl(x)\subseteq \cl(X)$. Hence $Z\subseteq \cl(X)$. Altogether $Z=\cl(X)$.   
\end{proof}

By Proposition \ref{P_degenerated}   a proper, totally  degenerated closure operator is like the one in Example \ref{Ebasic}(3): there is an equivalence relation on $P$ and a linear order on the classes  such that: for every $X\subseteq P$  $\cl(X)$ is  the union of all classes which are not greater than the class of some element of $X$.

\smallskip Next we recall the notion of a $\cl$-free sequence. Let $(P,\cl)$ be a  closure system and let $\mathbb I=(I,<)$ be a linear order.  
  The sequence    $(a_i\,|\,i\in I)$  of  elements of $P\smallsetminus\cl(\emptyset)$  is  called  {\em $\cl$-free} if $a_i\notin \cl(\bar a_{<i})$ holds for all $i\in I$.

\begin{prop}\label{P_totdeg_cl_iplies_invariant}Suppose that $\cl$ is a totally degenerated closure operation on $P$ and let $a,a_1,\ldots,a_n\in P\smallsetminus \cl(\emptyset)$. 
\begin{enumerate}[(i)]
\item The following conditions are all equivalent:

(a) $(a_1,a_2)$ is $\cl$-free; \ \ (b) $\eps(a_1)<_{\cl}\eps(a_2)$; \ \  (c) $a_1<_{\cl}\eps(a_2)$; \ \ (d) $\eps(a_1)<_{\cl}a_2$.

\item $(a_1,\ldots ,a_n)$ is $\cl$-free is equivalent to either of the following: \\
(a)  \ $\eps(a_1)<_{\cl}\ldots<_{\cl}\eps(a_n)$ \ (b)  \ $\pi(a_1)<_\cl\ldots <_\cl\pi(a_n)$ \ \ (c) \ $\cl(a_1)\subsetneq \ldots \subsetneq \cl(a_n)$. 
\item Every maximal $\cl$-free sequence is of order-type $\mathbb P_{\cl}$.  

\item Suppose that $\leq$ partially orders $P\smallsetminus\cl(\emptyset)$ such that all  $\cl$-free sequences are strictly increasing. Then $\eps(a)$ is $\leq$-convex and closed under $\leq$-incomparability (in $P\smallsetminus\cl(\emptyset)$).
\end{enumerate}
\end{prop}
\begin{proof}
\smallskip (i) By Proposition \ref{P_degenerated}(ii) we have that $(a_1,a_2)$ is $\cl$-free iff $\cl(a_1)\subsetneq\cl(a_2)$ iff $\eps(a_1)<_\cl\eps(a_2)$. Therefore (a)$\Leftrightarrow$(b) holds. (b)$\Rightarrow$(c) and (b)$\Rightarrow$(d) are obvious. To prove (c)$\Rightarrow$(b) assume that $a_1<_\cl\eps(a_2)$. Then by definition  $\cl(a_1)\subsetneq\cl(a_2)$, hence by Proposition \ref{P_degenerated}(ii) $\eps(a_1)<_\cl\eps(a_2)$, which proves (c)$\Rightarrow$(b). Similarly (d)$\Rightarrow$(b) holds.

\smallskip (ii) This is an easy consequence of part (i) and  Proposition \ref{P_degenerated}.

\smallskip
(iii) Note that for a linear order $\mathbb I=(I,<)$, a sequence of elements in $P\smallsetminus \cl(\emptyset)$, $(a_i\mid i\in I)$, is $\cl$-free if and only if $(\cl(a_i)\mid i\in I)$ is  $\subsetneq$-increasing. Therefore, a sequence $(a_i\mid i\in I)$ is maximal $\cl$-free iff $\{\cl(a_i)\mid i\in I\}= \{\cl(x)\mid x\in P\smallsetminus \cl(\emptyset)\}$, in which case, by Proposition \ref{P_degenerated}(ii), $\mathbb I\cong\mathbb P_\cl$ holds.

\smallskip 
(iv)  Suppose that $x\leq x'$ both belong to $\eps(a)$; then $\eps(x)=\eps(x')=\eps(a)$. Assume that   $y\notin\cl(\emptyset)$.  $y\notin\eps(a)$ implies   either $\eps(a)<\eps(y)$ or $\eps(y)<\eps(a)$, so   either $x'<y$ or $y<x$ holds and $x<y<x'$  cannot hold. Hence $x<y<x'$ implies $y\in\eps(a)$ and $\eps(a)$ is convex. That $\eps(a)$ is closed under incomparability is proved similarly. 
\end{proof}

Till the end of the section  we will discuss closure operations appearing in the model-theoretic context of regular types from \cite{PT}. So from now on we operate in $\Mon$.  

\begin{defn}\label{D_clp}
Let $(N,...)$ be any first-order structure and assume that $p\in S_1(N)$  does not split over $A\subset N$.  For $X\subseteq N$ define the {\em $p$-closure} operator on $N$: \ \ $\cl_p^A(X)=\{x\in N\mid x\nvDash p\mid AX\}\,.$
\end{defn}

It is easy to see that $\cl_p^A$ satisfies  Monotonicity and Finite character. Therefore, $\cl_p^A$ is closure operation   iff it satisfies Transitivity. Usually we will consider the restriction of $\cl^A_{p}$ to  $p_{\strok A}(N)$, it will be denoted also by $\cl^A_p$.  If $(I,<)$ is a linear order then    $(a_i\mid i\in I)\subseteq p_{\strok A}(N)$ is a {\em weak Morley sequence} in $p$ over $A\subset N$ if $a_i\models p_{\strok A\bar a_{<i}}$ for all $i\in I$.  

\begin{rmk}\label{Rmorley_equals_clfree}
Assume that $\cl_p^A$ is a closure operation on $p_{\strok A}(N)$. Then the $\cl_p^A$-free sequences are precisely the  weak Morley sequences in $p$ over $A$. Indeed, for  any  $a_1\in p_{\strok A}(N)$: \ $a_2\models p_{\strok Aa_1}$ \  iff \ $a_2\notin\cl_p^A(a_1)$.
\end{rmk}

The next theorem is a bit more general version of Theorem \ref{T_regular_implies_totdeg_uvod}.  

\begin{thm}\label{T_cl_implies_order}
Suppose that  $p\in S(N)$ does not split over $A\subset N$ and that   $\cl^A_{p}$ is  a closure operation on $p_{\strok A}(N)$. Further, assume that $N$ is $(|T|+|A|)^+$-saturated,  that $(a,b)\in N^2$ is a weak Morley sequence in $p$ over $A$,  and that $\tp(a,b/A)\neq\tp(b,a/A)$. Then:
\begin{enumerate}[(i)]
\item There is an $A$-definable partial order such that any weak Morley sequence in $p$ over $A$ is strictly increasing.

\item $\cl^A_p$ is a totally degenerated operation on $p_{\strok A}(N)$. 

\item For any $X\subseteq N$ the order type of any maximal  in $X$  weak Morley sequence in $p$ over $A$ does not depend on the particular choice of the sequence.
\end{enumerate}
\end{thm}
\begin{proof}(i)  The proof of Theorem 3.1 from \cite{PT} goes through.  For the sake of completeness we include it.  $\tp(a,b/A)\neq\tp(b,a/A)$ implies  $a\in\cl_p(b)$ and  $b\notin\cl_p(a)$. Then there is   $\phi(x,y)\in\tp(a,b/A)$ such that $\phi(a,x)\in p$ and \,$\phi(x,b)\notin p$. Moreover, after replacing it by $\phi(x,y)\wedge\neg\phi(y,x)$ we may assume that $\phi$ is asymmetric: $\models \phi(x,y)\Rightarrow \neg\phi(y,x)$.  By non-splitting, $\phi(x,a)\notin p$. We {\em claim}:
$$p_{\strok A}(x)\cup\{\phi(x,a)\}\vdash\phi(x,b)\,.$$ 
To prove it, assume that $d$ realizes $p_{\strok A}$ and satisfies 
\ $ \phi(x,a)$. Then $\models \phi(d,a)$  implies $d\in\cl_p(a)$ which together with $b\notin \cl_p(a)$ implies $b\notin\cl_p(d)$. Thus $(d,b)$ is a weak Morley sequence in $p$ over $A$ and $\models \phi(d,b)$ follows. Since $N$ is $|A|^+$-saturated,  the claim follows. 

By compactness, there is a formula $\theta(x)\in p_{\strok A}(x)$ such that $\models (\theta(t)\wedge\phi(t,a))\Rightarrow\phi(t,b)$\,. Therefore, $\phi(N,x)\subseteq \phi(N,y)$ defines a quasi-order $\leq$ on $\theta(N)$ such that $a\leq b$ holds. The asymmetry of $\phi$ implies $a\notin\phi(N,a)$ so  $a<b$ holds. Then it holds for any two consecutive elements in a weak Morley sequence, so the sequence is strictly increasing.

\smallskip (ii) 
First note that, by non-splitting,  $c<x\in p(x)$ for any $c$ realizing $p$. Thus $\cl_p(c)$ contains all the realizations of $p_{\strok A}$ which are either smaller or incomparable to $c$.  We prove that for any finite, non-empty $B\subset p_{\strok A}(N)$ there exists $d\in B$ such that $\cl_p(B)=\cl_p(b)$. Indeed, take $d\in B$ to be a maximal element (in $B$). Consider the formula $d<x$. It belongs to $p(x)$ and no element of $B$ satisfies it. Hence, no element of $B$ realizes $p_{\strok Ad}$, so $B\subseteq \cl_p(d)$ holds.  Therefore, $\cl_p$ is totally degenerated. 

\medskip 
(iii) Consider the restriction of $\cl^A_p$ to $A\,X$. It is totally degenerated so, by Proposition \ref{P_totdeg_cl_iplies_invariant}, any two maximal $\cl^A_p$-free sequences have the same order type. By Remark \ref{Rmorley_equals_clfree}  $\cl^A_p$-free sequences are weak Morley sequences in $p$ over $A$. The conclusion follows. 
\end{proof}

\begin{defn}In the context of Theorem \ref{T_cl_implies_order}, the order type of any maximal $\cl^A_p$-free sequence of elements of $X\subseteq N$ is called the {\em $p$-invariant of $X$ over $A$} and is denoted by   $\Inv_{p,A}(X)$. 
\end{defn}

Theorem \ref{T_cl_implies_order} has a strong assumption that $N$ is sufficiently saturated and a weak assumption that $\cl^A_p$ is a closure operator on $p_{\strok A}(N)$ (instead of on the whole of $N$). We can relax the first and strengthen the second assumption and still deduce the same conclusion. This is done in the next theorem; since it will not be used further in the text  the details of the proof are left to the reader.  

\begin{thm} 
Suppose that  $p\in S(N)$ does not split over $A\subset N$ and that   $\cl^A_{p}$ is  a closure operation on $\phi(N)$ for some $\phi(x)\in p_{\strok A}(x)$. Further, assume that  $(a,b)\in N^2$ is a weak Morley sequence in $p$ over $A$  and that $\tp(a,b/A)\neq\tp(b,a/A)$. Then:
\begin{enumerate}[(i)]
\item There is an $A$-definable partial order such that any weak Morley sequence in $p$ over $A$ is strictly increasing.

\item $\cl^A_p$ is a totally degenerated operation on $p_{\strok A}(N)$. 

\item For any $X\subseteq N$ the order type of any maximal  in $X$  weak Morley sequence in $p$ over $A$ does not depend on the particular choice of the sequence. 
\end{enumerate}
\end{thm}
\begin{proof}The proof of part (i) follows (a part of) the proof of Theorem 6.1 from \cite{PT}. (ii)-(iii) then follow as in the proof of Theorem \ref{T_cl_implies_order}.
\end{proof}

\section{Regularity}\label{X_regular}

A global non-algebraic type $\mathfrak p$ is said to be {\em regular over $A$} if it is invariant over $A$  and: 
\begin{center}
  for all $B\supseteq A$ and $a\models\mathfrak p_{\strok A}$: \ either \ $a\models \mathfrak p_{\strok B}$ \ or \ $\mathfrak p_{\strok B}\forces \mathfrak p_{\strok Ba}$ holds.
  \end{center}
An invariant global type is {\em regular} if it is regular over some small set. This definition is from  \cite{Tgs} where a minor imprecision from \cite{PT} was corrected. In order to prove Theorems   \ref{T_regular_implies_totdeg_uvod} and \ref{T_omit_uvod} we will not need the full power of the regularity assumption.

\begin{defn}\label{Dweakly_regular}
A global $A$-invariant type is {\em weakly regular over $A$} if it is $A$-invariant and for all $X\subset \frak p_{\strok A}(\Mon)$ and  $a\models\mathfrak p_{\strok A}$: \ either \ $a\models \mathfrak p_{\strok AX}$ \ or \ $\mathfrak p_{\strok AX}\forces \mathfrak p_{\strok AXa}$ holds. $\frak p$ is {\em weakly regular} if it is weakly regular over some small $A$.
\end{defn}

An alternative way of defining weak regularity over $A$ is that $\cl_\frak p^A$ is a closure operator on $\frak p_{\strok A}(\Mon)$. 

\begin{rmk}\label{R_reg_implies_closure}
Suppose that  $\frak p$ is weakly regular over $A$.

(i) $\cl^A_p$ is a closure operator on $\frak p_{\strok A}(\Mon)$. As we noted after Definition \ref{D_clp} it suffices to verify the transitivity. So assume that  $X\subset \frak p_{\strok A}(\Mon)$ is small, that $a,b$ realize $\frak p_{\strok A}$, $a\in\cl^A_\frak p(X)$ and $b\in\cl_\frak p^A(Xa)$. Then  $a\nmodels \frak p_{\strok AX}$ and $b\nmodels \frak p_{\strok AXa}$ follow from the definition of $\cl^A_\frak p$ and, by definition of weak regularity over $A$, they imply $b\in\cl^A_\frak p(X)$. Therefore $\cl^A_\frak p(X)=\cl^A_\frak p(Xa)$. By induction, using Finite Character, the transitivity follows.

\smallskip
(ii) Morley sequences in $\frak p$ over $A$ are precisely the $\cl^A_\frak p$-free sequences. In particular,  there are $\cl^A_\frak p$-free sequences of size $|\Mon|$.

\smallskip
(iii) Clearly, the regularity over $A$ implies weak regularity over $A$. Actually, the regularity of $\frak p$ over $A$ is equivalent to $\cl^B_\frak p$ being a closure operation on $\frak p_{\strok B}(\Mon)$ for every $B\supseteq A$. However, we will not need  this fact further in the text, nor we will need to change the base set $A$ over which $\frak p$ is regular in $\cl^A_\frak p$; we will consider only $\cl^A_\frak p$ for a fixed $A$.
\end{rmk} 

\begin{defn}\label{Dassym}
An invariant  type $\frak p$ is {\em $A$-asymmetric} if it is invariant over $A$ and $\tp(a,b/A)\neq\tp(b,a/A)$ holds for some (any) Morley sequence  $(a,b)$ in $\frak p$ over $A$. 
\end{defn}

By a (weakly) regular, $A$-asymmetric type we will mean an $A$-asymmetric type which s (weakly) regular over $A$. We can now easily prove an extended version of Theorem \ref{T_regular_implies_totdeg_uvod}.
    
\begin{thm}\label{T_regular_implies_totdeg} Suppose that $\frak p$ is weakly regular and $A$-asymmetric.
\begin{enumerate}[(i)]
\item There is an $A$-definable partial order such that any Morley sequence in $p$ over $A$ is strictly increasing.

\item $\cl^A_p$ is a totally degenerated closure operation on $\frak p_{\strok A}(\Mon)$. 

\item For any   $X\subseteq \Mon$ the order type of any maximal Morley sequence (in $p$ over $A$) consisting of elements of $X$   does not  depend   on the particular choice of the sequence. 
\end{enumerate}
\end{thm}
\begin{proof}
$\cl^A_{\mathfrak p}$ is a closure operation on $\mathfrak p_{\strok A}(\Mon)$ by Remark \ref{R_reg_implies_closure}, and $A$-asymmetry of $\frak p$ over $A$ guaranties that all the assumptions of   Theorem  \ref{T_cl_implies_order} are satisfied. (i)-(iii) follow. 
\end{proof}

We will refer to the order satisfying condition of part (i) of the  theorem  as to a witness of  $A$-asymmetry of  $\frak p$;  note that the witnessing order may not be unique.  
 
\begin{nota}
  Suppose that $\mathfrak p$ is  weakly regular and $A$-asymmetric, and that $X\subset \Mon$; let $p=\frak p _{\strok A}$. Then, by Theorem \ref{T_regular_implies_totdeg}(ii) $\cl^A_\frak p$ is a totally degenerated closure operation on $p(\Mon)$. We adapt the notation from Section \ref{X_closure} to that context.  
\begin{enumerate}
\item $\cl_\frak p$ will denote $\cl^A_\frak p$ \ (the meaning of the set $A$  will always be clear from the context, and we will not consider   operations $\cl^B_\frak p$ for $B\supset A$).

\item $\eps_\frak p(a):=\{b\in p(\Mon)\,|\, \cl_\frak p(a)=\cl_\frak p(b)\}=\{b\in p(\Mon)\,|\,b\nmodels \frak p_{\strok Aa} \ \mbox{and}  \  a\nmodels \frak p_{\strok Ab}\,\}.$ \ In other words, $\eps_\frak p(a)$ is the set of all $b\models \frak p_{\strok A}$ such that $\{a,b\}$ cannot be ordered into a Morley sequence in $\frak p$ over $A$. 

\item $\Lin_A(\frak p):= \{\eps_\frak p(a)\,|\,a\in p(\Mon)\}$; it corresponds to $P_\cl$ from \ref{P_totdeg_cl_iplies_invariant} and  is naturally (linearly) ordered by: \ $\eps_\frak p(a)<\eps_\frak p(b)$ \ iff \ $(a,b)$ is a Morley sequence in $\frak p$ over $A$. 

\item $\pi_\frak p$ is the projection map from $p(\Mon)$ onto $\Lin(\frak p)$ defined by $\pi_\frak p(a)=\eps_\frak p(a)$

\item     $\Inv_{\mathfrak p,A}(X)$  denotes the linear-order type of any maximal Morley sequence (in $p$ over $A$) consisting of elements of $X$. If no such sequence exists then $\Inv_{\mathfrak p,A}(X)=\emptyset$.
\end{enumerate} 
\end{nota}

Next we note important corollaries of Theorem \ref{T_cl_implies_order} and Proposition \ref{P_totdeg_cl_iplies_invariant} which will be used  further in the text without specific mentioning. 

\begin{cor}\label{cor_order_and_morley} Suppose that $\frak p$ is weakly regular and $A$-asymmetric witnessed by $<$, and that  $a,a_1,...,a_n$ realize $\mathfrak p_{\strok A}$. 
\begin{enumerate}[(i)]
\item The following conditions are all equivalent:

(a) $(a_1,a_2)$ is a Morley sequence in $\frak p$ over $A$; \ \ (b) $\pi_\frak p(a_1)<\pi_\frak p(a_2)$; \ \ (c) $\eps_\frak p(a_1)< \eps_\frak p(a_2)$;\\ (d) $a_1< \eps_\frak p(a_2)$; \ \ (e) $\eps_\frak p (a_1)< a_2$.

\item $(a_1,...,a_n)$ is   a Morley sequence in $\frak p$ over $A$ if and only if either of the following holds: \\
(a)  \ $\eps_\frak p(a_1)<  ...< \eps_\frak p(a_n)$; \ \ (b)  \ $\pi_\frak p(a_1)< ...<\pi_\frak p(a_n)$; \ \ (c) \ $\cl_\frak p(a_1)\subsetneq  ...\subsetneq \cl_\frak p(a_n)$\,.
\end{enumerate}
\end{cor}

\begin{cor}  Suppose that $\frak p$ is weakly regular and $A$-asymmetric.
\begin{enumerate}[(i)]
\item Suppose that $\leq$ partially orders $P$ such that all  $\cl$-free sequences are strictly increasing. Then $\eps_\frak p(a)$ is $\leq$-convex and closed under $\leq$-incomparability in $(\frak p_{\strok A}(\Mon),\leq)$: if $b\models \mathfrak p_{\strok A}$ is $\leq$-incomparable to an element of $\eps_\frak p(a)$ then $b\in \eps_\frak p(a)$. 
\item There is a canonical linear order on $\Lin_A(\frak p)$ which agrees with any witness of asymmetric regularity; it will be denoted by $\mathbb L_A(\frak p)$.  $\mathbb L_A(\frak p)$ is isomorphic to $ (\Lin_A(\frak p),<)$, where $<$ is any linear order witnessing $A$-asymmetric regularity and here it is extended to   subsets of $\frak p_{\strok A}(\frak C)$.
\end{enumerate}
\end{cor}

We recall the definition of strongly regular types from \cite{PT}.  

\begin{defn}   We say that  $(\mathfrak p(x),\phi(x))$ is \emph{strongly regular} if for some small $A$ over which $\frak p$ is invariant, $\phi(x)\in\mathfrak p_{\strok A}$ and: 
\begin{center}
for all $B\supseteq A$ and $a\in\phi(\Mon)$: \ either \ $a\models \mathfrak p_{\strok B}$ \ or \ $\mathfrak p_{\strok B}\forces \mathfrak p_{\strok Ba}$ holds.
\end{center}
$\mathfrak p$ is {\em strongly regular} if there exists  $\phi(x)\in\mathfrak p$ such that $(\mathfrak p(x),\phi(x))$ is strongly regular. 
\end{defn}

Here are  some examples of regular and strongly regular types. They are based on well-known structures which have certain elimination of quantifiers, so the details are left to the reader.

\begin{exm}(1) Let $(\Mon,<)$  be a dense linear order without end points and let $\mathfrak p$ be the type of an infinitely large element. Then $(\mathfrak p(x), x=x)$ is strongly regular,  $\mathfrak p$ is definable and $\emptyset$-asymmetric  witnessed by $<$. Furthermore, if $q\in S_1(M)$ then $q$ has two $M$-invariant globalizations: $\frak q_l$ which contains $\{x<a\,|\,a\models q\}$, and $\frak q_r$ which contains $\{x>a\,|\,a\models q\}$. Both $(\frak q_l,x=x)$ and $(\frak q_r,x=x)$ are strongly regular. 

\smallskip
(2) There are  examples similar to (1) with discrete orders. Let $(\Mon,<)$ be a monster of $(\mathbb N,<)$. Consider the type $q\in  S_1(\mathbb N)$ of an infinitely large element. It has two $\mathbb N$-invariant globalizations: $\frak q_l$ which is the coheir of $q$ and $\frak q_r$ which is the heir. Both $(\frak q_l,x=x)$ and $(\frak q_r,x=x)$ are strongly regular. 

\smallskip
(3) Let $(\Mon,+,<,0)$ be the monster model of the ordered group of rationales and let $\mathfrak p$ be the type of an infinitely large element. Then $(\mathfrak p,x=x)$ is strongly regular,  $\mathfrak p$ is definable and $\emptyset$-asymmetric  witnessed by $<$. 

\smallskip
(4) In \cite{PT} it is shown that `generic' types of  minimal and quasi-minimal groups are definable and that their global heirs are strongly regular. Example 5.1 there describes a quasiminimal field whose generic heir  is $\emptyset$-invariant. Examples of minimal groups with $\emptyset$-asymmetric global heir can be found in  \cite{KTW} and \cite{GK}.  
\end{exm}

\begin{defn}\label{D_convextype}
A weakly regular $A$-asymmetric type $\frak p$ is {\em convex over $A$} if there is   $\leq$ witnessing the $A$-asymmetric regularity   such that $\frak p_{\strok A}(\Mon)$ is a $\leq$-convex subset of $\Mon$. 
\end{defn}

The following example shows that non-convex regular types  exist.

\begin{exm}\label{Ex_nonconvex}Let $\mathcal L=\{<\}\cup \{P_i\,|\,i\in\omega\}$ where each $P_i$ is unary and $<$ is binary. Consider $(\mathbb Q,<,P_i)_{i\in \omega}$ where each $P_i$'s partition the universe into everywhere dense pieces, and let $(\Mon,<,P_i)_{i\leq n}$ be the monster model. Let $\mathfrak p_n$ be the type of an infinitely large element satisfying $P_n(x)$. Then  $(\mathfrak p_n(x),x=x)$ is  strongly regular, and  $\frak p_n$ is $\emptyset$-asymmetric and definable. The locus of the restriction of   $\mathfrak p_n$  to $\emptyset$ is not $<$-convex.  However, we can modify $<$ by restricting it to $P_n(\Mon)$ so that the locus becomes convex. Hence $\frak p_n$ is convex which is, by Lemma \ref{L_SR_is_convex},   the case with any strongly regular type.

Now, let let $\mathfrak p$ be the type of an infinitely large element containing $\{\neg P_i(x)\,|\,i\in\omega\}$. Then $\mathfrak p$ is regular over $\emptyset$ and asymmetric  (witnessed by $<$).   $ \mathfrak p_{\strok \emptyset} (\Mon)$ is not convex with respect to any partial order which orders Morley sequences in $\mathfrak p$ over $\emptyset$ increasingly. 
 \end{exm}

In the next lemma we prove that asymmetric strongly regular types are convex, which  is a significant property used in almost all the results further in the text. Actually, in most of the results  the convexity assumption will suffice.   

\begin{lem}\label{L_SR_is_convex}
If $(\mathfrak p(x),\phi(x))$ is strongly regular and  $A$-asymmetric then $\frak p$ is convex over $A$. 
\end{lem}
\begin{proof}
Let $<$ be (any) witness of asymmetric regularity. Replace it by its restriction to  $\phi(\Mon)$.  We will prove that after this modification $\frak p_{\strok A}(\Mon)$ is a  convex subset of $\Mon$.  Suppose that $a,c\models \frak p_{\strok A}$ and $a<b<c$; we will prove $b\models \frak p_{\strok A}$. Clearly $b\in\phi(\Mon)$. Take $d\models\mathfrak p_{\strok Abc}$. Then $a<b<c<d$ and hence $b<x\in\mathfrak p_{\strok Ab}$. If we assume that $b\nmodels \frak p_{\strok A}$, from the strong regularity we have that $\frak p_{\strok A}\vdash\mathfrak p_{\strok Ab}$. Since $a\models \frak p_{\strok A}$, we get that $a\models \mathfrak p_{\strok Ab}$  and thus $b<a$. A contradiction. Hence $b\models \frak p_{\strok A}$.
\end{proof}

\section{$\eps_\frak p$-classes of weakly regular types}\label{X_symm_closure}

\begin{assu}Throughout the section we fix a weakly regular, $A$-asymmetric type $\frak p$ and $<$ witnessing the asymmetric regularity. By $p$ we denote $\frak p_{\strok A}$.
\end{assu} 
We are interested in closer descriptions of  $\eps_\frak p(a)$ when $a\models p$. The first observation is that $\eps_\frak p(a)$ is relatively $\bigvee$-definable over $Aa$ within $p(\Mon)^2$. Indeed, for $a,b\models p$ we have that $a\notin \eps_\frak p(b)$ holds if and only if at least one of $(a,b)$ and $(b,a)$ is a Morley sequence in $\frak p$ over $A$.  Since `being a Morley sequence in $\frak p$ over $A$' is   type-definable   over $A$, so is $\{(a,b)\,|\,a\notin\eps_\frak p(b)\}$. Therefore the complement is relatively $\bigvee$-definable and is defined by an infinite disjunction. We will prove in Corollary \ref{C_scl_rel_definable_reg} that the disjuncts can be taken in a specific form. That form will be adequate for applying the Omitting Types Theorem in Section \ref{X_real_invariants}.  

\begin{defn}\label{D_pbounded} 
\begin{enumerate}[(i)]

\item $D\subset \Mon$ is {\em $\frak p$-bounded over $A$} if it meets $p(\Mon)$ and there are $a,b$ realizing $p$  such that \  $a<'D\cap \frak p(\Mon) <'b$ \ holds  for some $<'$ witnessing $A$-asymmetric regularity of $\frak p$. 

\item $\psi(x)$ is {\em $\frak p$-bounded over $A$}  if  $\psi(\Mon)$ is $\frak p$-bounded over $A$. 

\item $\psi(x)$ is  {\em strongly $\frak p$-bounded over $A$ with respect to $<$} if it is consistent with $p(x)$ and there are   $a,b$ realizing $p$  such that \  $a<\psi(\Mon) <b$. 
\end{enumerate}
\end{defn}

\begin{rmk}
(i) It is easy to see that $D$ is $\frak p$-bounded if and only if $\pi_\frak p(D)$ is bounded from both sides in $\Lin_A(\frak p)$: if \  $a<'D\cap \frak p(\Mon) <'b$ \ and we choose $a',b'$ realizing $p$ such that $\eps_\frak p(a')<\eps_\frak p(a)$ and $\eps_\frak p(b)<\eps_\frak p(b')$, then $\eps_\frak p(a')$ and $\eps_\frak p(b')$ are bounds for $\pi_\frak p(D)$. 
 Since the order on $\Lin_A(\frak p)$ does not depend on the  particular choice 
 of the ordering relation witnessing $A$-asymmetric weak regularity, neither does $\frak p$-boundedness of $D$. 

\smallskip (ii) Strong boundedness may depend on the choice of $<'$. However in Lemma \ref{L_p_bounded_fla}(i) we will see that a minor, definable correction of any ordering  may witness strong boundedness. Usually the meaning of $<'$ will be clear from the context and will be omitted. For example, in this section  we will consider only strong boundedness w.r. to our fixed $<$ and we will simply say that a formula is strongly $\frak p$-bounded.
\end{rmk} 

\begin{lem}\label{L_p_bounded_fla} Suppose that  $\theta(x,\bar y)$ is over $A$, that $\theta(x,\bar b)$ is $\frak p$-bounded over $A$, and that $c,c'$ realizing $p$ are such that $c<\theta(\Mon,\bar b)\cap p(\Mon)<c'$.  
\begin{enumerate}[(i)]
\item There exists $\phi(x)\in p(x)$  such that \ $c<\phi(\Mon) \cap \theta(\Mon,\bar b)<c'$. In particular $\phi(x)\wedge\theta(x,\bar b)$ is strongly $\frak p$-bounded (w.r. to $<$). 

\item 
Moreover, if $\frak p$ is  convex over $A$ and $a\models p$   satisfies $\theta(x,\bar b)$, then $\phi(x)\in p(x)$ can be chosen such that $\phi(x)\wedge\theta(x,\bar b)$ witnesses $a\in\Sem_A(\bar b)$.
\end{enumerate}
\end{lem}
\begin{proof} (i) \ Since $\theta(x,\bar b)$ is $\frak p$-bounded we have: \
 $p\cup\{\theta(x,\bar b)\}\vdash c<x<c'$. By compactness there is $\phi(x)\in p(x)$ such that \ $\models (\phi(x)\wedge\theta(x,\bar b))\Rightarrow c<x<c'$.

\smallskip (ii) Choose $<'$ such that $p(\Mon)$ is $<'$-convex in $\Mon$.  By the proof of  part (i), applied  to $<'$ and $\theta(x,\bar b)$, we have $\models (\phi(x)\wedge\theta(x,\bar b))\Rightarrow c<'x<'c'$. Since $p$ s $<'$-convex that implies \ $\phi(x)\wedge\theta(x,\bar b)\forces p(x)$. Hence $\phi(x)\wedge\theta(x,\bar b)$  witnesses $a\in\Sem_A(\bar b)$.  
\end{proof}

We now turn to   $\frak p$-bounded formulas $\sigma(a,x)$ when $a\models p$. 

 \begin{lem}\label{L_str_bndd_morley} Suppose that $\sigma(x,y)$ is over $A$,    $a\models p$, and that $\sigma(a,x)$ is consistent with $p(x)$. Then  the following conditions are equivalent:
\begin{enumerate}[(1)]
\item $\sigma(a,x)$ is strongly $\frak p$-bounded over $A$.

\item   $c_1<\sigma(a,\Mon)<c_2$ holds for all Morley sequences $(c_1,a,c_2)$ (in $\frak p$ over $A$).
\end{enumerate}  
\end{lem}
\begin{proof}  (2)\,$\Rightarrow$\,(1) follows immediately from the definition. To prove the other direction, assume that $\sigma(a,x)$ is strongly $\frak p$-bounded over $A$. Let $c,c'$ realize $p$ be such that \  $c <\sigma(a,\Mon)<c'\,$. Choose $c_1,c_2$ realizing $p$ such that $(c_1,c)$ and $(c',c_2)$ are Morley sequences. Then
  $(c_1,a,c_2)$ is a Morley sequence   in $\frak p$ over $A$ for and   \   $c_1<\sigma(a,\Mon)<c_2$\,.  
Since that holds for one Morley sequence $(c_1,a,c_2)$ it holds for all. The proof of the lemma is  complete.
\end{proof}

\begin{cor}\label{C_strong_subset_scl}
Suppose that $a\models p$ and that $\sigma(a,x)$ is strongly $\frak p$-bounded over $A$. Then 
$\sigma(a,\Mon)\cap p(\Mon)\subseteq \eps_\frak p(a)$. 
\end{cor}
\begin{proof}
By Lemma  \ref{L_str_bndd_morley} $\sigma(a,\Mon)\cap p(\Mon)$ cannot contain an element $a'$ such that at least one of $(a,a')$ and $(a',a)$ is a Morley sequence; otherwise, $a'$ would be a member of some Morley  sequence $(c_1,a,c_2)$.  Therefore  $\sigma(a,\Mon)\cap p(\Mon)$ contains only elements of $\eps_\frak p(a)$. 
\end{proof}

The next lemma emphasises the role of semi-isolation for convex types.

\begin{lem}\label{Lscl_rel_definable}   Suppose that  $a\models p$.  
\begin{enumerate}[(i)]
\item For all  $b\in\eps_\frak p(a)$  there is a symmetric formula $\sigma(x,y)\in\tp(a,b/A)$ such that $\sigma(a,y)$ is strongly $\frak p$-bounded over $A$.  

\item 
If $p(\Mon)$ is $<$-convex then  $ \eps_\frak p(a)\subseteq\Sem_A(a)$.

\item 
If $p(\Mon)$ is $<$-convex in $\Mon$ and $b\in\eps_\frak p(a)$ then there is a symmetric formula $\sigma(x,y)\in\tp(a,b/A)$ witnessing $a\in \Sem_A(b)$ and $b\in \Sem_A(a)$,  such that $\sigma(a,y)$ is strongly $\frak p$-bounded over $A$.
\end{enumerate}
\end{lem}
\begin{proof}
(i) Assume that $b\in\eps_\frak p(a)$. Then neither $(a,b)$ nor $(b,a)$ is a Morley sequence in $\frak p$ over $A$, so there is $\sigma'(x,y)\in\tp(a,b/A)$ such that $\sigma'(x,b)\notin\mathfrak p_{\strok Ab}$ and $\sigma'(a,y)\notin\mathfrak p_{\strok Aa}$. Consider the formula $\sigma''(x,y):= \sigma'(x,y)\vee \sigma'(y,x)$. It is symmetric and $\sigma''(x,y)\in\mathrm{tp}(a,b/A)$. From $\sigma'(x,b)\notin\frak p_{\strok Ab}$, by $A$-invariance of $\frak p$, we have $\sigma'(x,a)\notin\frak p_{\strok Aa}$, which together with $\sigma'(a,y)\notin\frak p_{\strok Aa}$ implies $\sigma''(a,y)\notin\frak p_{\strok Aa}$. Similarly, $\sigma''(x,b)\notin\frak p_{\strok Ab}$.

Assume that $b'$ realizes $p$ and $\models\sigma''(a,b')$.  Then $\sigma''(a,y)\notin\mathfrak p_{\strok Aa}$ implies $b'\nmodels \frak p_{\strok Aa}$. Since, by invariance, $\sigma''(x,b')\notin\mathfrak p_{\strok Ab'}$ we conclude $a\nmodels \mathfrak p_{\strok Ab'}$. Therefore, from $b'\nmodels\frak p_{\strok Aa}$ and $a\nmodels\frak p_{\strok Ab'}$ we conclude $b'\in\eps_\frak p(a)$. This proves  $\sigma''(a,\Mon)\cap p(\Mon)\subseteq\eps_\frak p(a)$, so $\sigma''(a,y)$ is $\frak p$-bounded. 

By Lemma \ref{L_p_bounded_fla}(i) we can now choose $\phi(y)\in p(y)$ such that $\phi(y)\wedge\sigma''(a,y)$ is strongly $\frak p$-bounded over $A$. Take $\sigma(x,y):=\phi(x)\wedge\phi(y)\wedge\sigma''(x,y)$. Clearly, $\sigma(x,y)\in\mathrm{tp}(a,b/A)$  is a symmetric formula and $\sigma(a,y)$ is strongly $\frak p$-bounded over $A$.

\smallskip (ii)   If $b\in\eps_\frak p(a)$ then, by part (i), there is a $\frak p$-bounded formula  $\sigma(a,y)\in\tp(b/Aa)$. By  Lemma \ref{L_p_bounded_fla}(ii)  $b\in\Sem_A(a)$. Hence $ \eps_\frak p(a)\subseteq\Sem_A(a)$.

\smallskip (iii) Let $b\in\eps_\frak p(a)$. Then $a\in\eps_\frak p(b)$ and, by parts  (i) and (ii) of the lemma, there are  symmetric formulas $\sigma_1(x,y),\sigma_2(x,y)\in\mathrm{tp}(a,b/A)$, such that $\sigma_1(a,y)$ is strongly $\frak p$-bounded over $A$ witnessing $b\in\Sem_A(a)$ and $\sigma_2(x,b)$ is strongly $\frak p$-bounded over $A$ witnessing $a\in\Sem_A(b)$. Now we can finish the proof of (iii) by taking $\sigma(x,y)= \sigma_1(x,y)\wedge \sigma_2(x,y)$.
\end{proof}

 Let $\mathcal B$ be the set of all symmetric formulas $\sigma(x,y)$ over $A$ such that $\sigma(a,y)$ is strongly $\frak p$-bounded for some (any) $a$ realizing $p$. Actually $\mathcal B$ depends on $\frak p$, $A$ and the choice of $<$,  but their meaning will  always be clear from the context.

\begin{lem}\label{C_scl_rel_definable_reg}     
$\eps_\frak p(a)=\bigcup_{\sigma(x,y)\in \mathcal B}\sigma(a,\Mon)\cap p(\Mon)$ \ \ holds \ \ for all $a\models p$.
\end{lem}
\begin{proof}
 If $\sigma(x,y)\in\mathcal B$ then $\sigma(a,y)$ is strongly $\frak p$-bounded over $A$ and, by Corollary \ref{C_strong_subset_scl}, $\sigma(a,\Mon)\cap p(\Mon)\subseteq\eps_\frak p(a)$. Therefore $\bigcup_{\sigma(x,y)\in \mathcal B}\sigma(a,\Mon)\cap p(\Mon)\subseteq \eps_\frak p(a)$.
The other inclusion is a consequence of Lemma \ref{Lscl_rel_definable}(i): any $b\in\eps_\frak p(a)$ satisfies  $\sigma(a,y)$ for some $\sigma(x,y)\in\mathcal B$.
\end{proof}

\section{Realizing invariants}\label{X_real_invariants} 

In this section we study possibilities for $\Inv_{\frak p,A}(M)$ when $M$ is a small model. We will prove Theorem \ref{T_omit_uvod}. Each part of the theorem is proved in a separate proposition: \ref{P_conv_simple_invariants},  \ref{T_nonsimple_anyinv} and \ref{T_simple_nonconvex}.

\begin{defn}
A weakly regular $A$-asymmetric type is {\em simple over $A$} if $x\in\eps_{\mathfrak p}(y)$ is a relatively definable (equivalence) relation on $\mathfrak p_{\strok A}(\Mon)^2$.
\end{defn}

For simple, weakly regular types $\Lin_A(\frak p)$ is type-definable in $\Mon^{eq}$: Suppose that $\frak p$ is simple over $A$ and that $\sigma_\frak p(x,y)$ relatively defines $x\in \eps_{\frak p}(y)$. $\sigma_\frak p$ defines an equivalence relation so  after replacing it by $(\sigma_\frak p(x,y)\wedge\phi(x)\wedge\phi(y))\vee x=y$ for an adequately (by compactness) chosen $\phi(x)\in \frak p$, we may assume that it defines an equivalence relation on the whole of $\Mon$. The classes of this relation corresponding to realizations of $\frak p_{\strok A}$ are precisely the elements of $\Lin_A(\frak p)$, so $\Lin_A(\frak p)$ is  type-definable in $\Mon^{eq}$. By Lemma \ref{Lscl_rel_definable}, we can choose $\sigma_\frak p(x,y)$ such that it is symmetric and $\sigma_\frak p(a,y)$ is strongly $\frak p$-bounded  with respect to $<$ (fixed in advance). 

\begin{prop}\label{P_conv_simple_invariants} Assume that $\frak p$ is a weakly regular,  $A$-asymmetric  type  that is both  convex and  simple over $A$, witnessed by $<$. If  $\Inv_{\mathfrak p, A}(M)$ has at least two elements for some model $M\supseteq A$, then it is a dense linear order (possibly with one or both ends)  
\end{prop}
\begin{proof} Choose a symmetric formula $\sigma_\frak p(x,y)$ relatively defining $\eps_\frak p$ such that $\sigma_\frak p(a,y)$ is strongly $\frak p$-bounded for any $a\models p$. Assume that $\Inv_{\mathfrak p, A}(M)$ has at least two elements.  Let  $b_1,b_2\in p(M)$ be such that $\eps_\frak p(b_1)<\eps_\frak p(b_2)$. It suffices to show that there is $b\in p(M)$ such that $(b_1,b,b_2)$ is a Morley sequence. Consider the `formula' \ $\sigma_\frak p(b_1,\Mon)< x<\sigma_\frak p(b_2,\Mon)$. It is satisfied in $\Mon$ by any $b'$ for which $(b_1,b',b_2)$ is a Morley sequence; hence it is consistent. Therefore, it is satisfied by some $b\in M$. Then $b_1<b<b_2$ so, by convexity, $b$ realizes  $\frak p_{\strok A}(x)$. Clearly, $(b_1,b,b_2)$ is a Morley sequence of elements of $M$. This completes the proof.
\end{proof}

\begin{lem}\label{non_simple_scl}Suppose that  $T$ and $A$ are countable  and that $\mathfrak p$ is weakly regular and $A$-asymmetric, witnessed by $<$. Denote $\frak p_{\strok A}$   by $p$. 
Assume that $\mathfrak p$ is not simple over $A$ and let $a\models p$. Then there exists a sequence $(\sigma_n(x,y))_{n<\omega}$ of symmetric formulas with parameters from $A$ such that the following conditions hold:
\begin{enumerate}
\item  Each  $\sigma_n(a,x)$ is strongly $\frak p$-bounded.
\item $(\sigma_n(\Mon, a) )_{n<\omega}$ is   strictly increasing under inclusion. 

\item For every $n<\omega$ there are $b',b''\in(\sigma_{n+1}(\Mon,a)\smallsetminus\sigma_n(\Mon,a))\cap p(\Mon)$ such that $b'<\sigma_n(\Mon,a)<b''$. 

\item $\eps_\frak p(a)=\bigcup_{n<\omega}\sigma_n(\Mon,a)\cap p(\Mon)$.
\end{enumerate}
\end{lem}
\begin{proof} From the countability of $T$ and $A$ we have $\mathcal B=\{\sigma_n(x,y)\mid n<\omega\}$. By replacing each $\sigma_n(x,y)$ by $\bigvee_{i\leq n}\sigma_i(x,y)$    we may assume that $(\sigma_n(x,y))_{n<\omega}$ is a sequence of symmetric formulas over $A$, such that $(\sigma_n(\Mon,a))_{n<\omega}$ is    increasing   and $\eps_\frak p(a)=\bigcup_{n<\omega}\sigma_n(\Mon,a)\cap p(\Mon)$.
We {\em claim} that for every $n<\omega$ there exist $m\geq n$ and $b''\in\sigma_m(\Mon,a)\smallsetminus\sigma_n(\Mon,a)$ such that $b''\models p$ and $\sigma_n(\Mon,a)<b''$. Otherwise $\sigma_n (\Mon,a)<x$ would relatively define  $\frak p_{\strok Aa}$ and  $\frak p$ would be simple over $A$. A contradiction.

Similarly, we can prove that for every $n<\omega$ there exist $m\geq n$ and $b'\in\sigma_m(\Mon,a)\smallsetminus\sigma_n(\Mon,a)$ such that $b'\models p$ and $b'<\sigma_n(\Mon,a)$. Therefore, we can  extract a  subsequence  of  $(\sigma_n(x,y))_{n<\omega}$ satisfying the desired conditions.
\end{proof}

  Keeping  the notation of the lemma  we note that the following   are equivalent for all $a,b\models p$:
 \begin{center}
  (1) $a< \sigma_n(\Mon,b)$, for all $n\in\omega$; \ \ (2)  $\sigma_n(\Mon,a)<b$, for all $n\in\omega$; \ \ (3) $(a,b)$ is a Morley sequence. 
\end{center}
The first condition means $a<\eps_\frak p(b)$ and the second describes $ \eps_\frak p(a)<b$, they are   equivalent to (3).

\begin{prop}\label{T_nonsimple_anyinv}
Suppose that $T$ and $A$ are countable and that $\mathfrak p$ is weakly regular, $A$-asymmetric and non-simple over $A$. Then for every countable,   linear order $\mathbb I$ there exists a countable model $M_{\mathbb I}\supseteq  A$ such that $\Inv_{\mathfrak p,A}(M_{\mathbb I})\cong \mathbb I$.
\end{prop}
\begin{proof}\setcounter{equation}{0}
Let $p=\mathfrak p_{\strok A}$ and let  $\mathbb I=(I,<_I)$ be a countable linear order. Choose a Morley sequence  $\bar a_I=(a_i\,|\,i\in I)$ in $\mathfrak p$ over $A$. Define
 \ \ $\Sigma(x):=p(x)\cup\bigcup_{i\in I}x\notin\eps_\frak p(a_i)\,.$ 
 
It suffices to show that $\Sigma(x)$ can be omitted in a model containing $A\bar a_I$. Otherwise, by   Omitting Types Theorem, there is a finite subtuple $\bar a$  of $\bar a_I$ and a consistent formula $\psi(x,\bar y)$ over $A$ such that $\psi(x,\bar a)\vdash\Sigma(x)$. Without loss of generality, assume that $\bar a=a_1...a_n$ where the order induced by $\mathbb I$ on indexes is natural. After slightly modifying $\psi$ we may assume that one of the following cases holds:
\begin{enumerate}
\item \ $\psi(x,\bar a)\vdash a_k<x<a_{k+1}$, \ for some $k< n$;
\item \ $\psi(x,\bar a)\vdash x<a_1$ ;
\item \ $\psi(x,\bar a)\vdash  a_{n}<x$ .
\end{enumerate}
Consider the first case and assume that  $\psi(x,\bar a)\vdash a_k<x<a_{k+1}$. Then $\psi(x,\bar a)\vdash \Sigma(x)$ implies 
$$\psi(x,\bar a)\vdash \eps_\frak p(a_k)<x<\eps_\frak p(a_{k+1})\,.$$
Since $\psi(x,\bar a)$ is consistent we deduce: 
\begin{equation}\label{eqy1}
\mbox{ $(a_1,\ldots ,a_k,b,a_{k+1},\ldots ,a_n)$ is a Morley sequence in $\frak p$ over $A$ \ \ if and only if \  \ $\models\psi(b,\bar a)$.}
\end{equation}
Then, by the remark preceding the theorem, we have:  $$p(x)\ \cup\  \{\sigma_n(\Mon,a_k)<x\,|\,n\in\omega\}\ \cup \ ''x<\eps_\frak p(a_{k+1})''\ \vdash\ \psi(x,\bar a)\, $$ where $\{\sigma_n(x,y)\,|\,n\in\omega\}$ satisfies the conclusion of Lemma \ref{non_simple_scl}.  By compactness   there is $m$ such that \ 
\begin{equation}\label{eqy2}
p(x)\ \cup\ \{\sigma_m(\Mon,a_k)<x\}\ \cup\ ''x<\eps_\frak p(a_{k+1})''\ \vdash\ \psi(x,\bar a)\,.
\end{equation} 
By Lemma \ref{non_simple_scl} again there is \ $b''\in \sigma_{m+1}(\Mon,a_k)\smallsetminus \sigma_m(\Mon,a_k)$ \ realizing $p$, such that $\sigma_m(\Mon,a_k)<b''$.  Then  $\models\sigma_{m+1}(b'',a_k)$ witnesses $b''\in\eps_\frak p(a_k)$, so $b''<\eps_\frak p(a_{k+1})$. Therefore $b''$ satisfies the left hand side of (\ref{eqy2}), so it satisfies $\psi(x,\bar a)$, too. By (\ref{eqy1})  $(a_1,,..,a_k,b'',a_{k+1},...,a_n)$ is a Morley sequence so $\eps_\frak p(a_k)<b''$. A contradiction. The proof in the first case is complete.

\smallskip 
Similarly, in the second case we get: 
\begin{center}
 $(b,a_1,\ldots,a_n)$ is a Morley sequence in $\frak p$ over $A$ \ \ if and only if \  \ $\models\psi(b,\bar a)\, .$ 
\end{center}
Then we apply compactness to  \ 
 $p(x)\cup  \{x< \sigma_n(\Mon,a_1)\,|\,n\in\omega\} \vdash\psi(x,\bar a)\, $, \ and reach  a contradiction in a similar way. The proof in the third case is similar. 
\end{proof}

\setcounter{equation}{0}
\begin{prop}\label{T_simple_nonconvex}
Suppose that $T$ and $A$ are countable and that $\mathfrak p$ is weakly regular, $A$-asymmetric and simple over $A$. If $\frak p$ is not convex over $A$ then for every countable,   linear order $\mathbb I$ there exists a countable model $M_{\mathbb I}\supseteq  A$ such that $\Inv_{\mathfrak p,A}(M_{\mathbb I})\cong \mathbb I$.
\end{prop}
\begin{proof}
Let $p=\mathfrak p_{\strok A}$, $\mathbb I=(I,<_I)$, $\bar a_I=(a_i\,|\,i\in I)$   and  $\Sigma(x):=p(x)\cup\bigcup_{i\in I}x\notin\eps_\frak p(a_i)\,$ be as in the proof of Proposition \ref{T_nonsimple_anyinv}. 
It suffices to show that $\Sigma(x)$ can be omitted in a model containing $A\bar a_I$. 
Otherwise, there are  $\bar a\subset\bar a_I$   and a consistent formula $\psi(x,\bar y)$ over $A$ such that $\psi(x,\bar a)\vdash\Sigma(x)$, where the order on $\bar a=a_1...a_n$ is natural. Again we have the same three cases as in the proof of Theorem \ref{T_nonsimple_anyinv} to consider. 

Assume that we are in the first case: 
 \ $\psi(x,\bar a)\vdash a_k<x<a_{k+1}$, \ for some $k< n$. \ Then:
\begin{equation}\label{Eq_l1}
\mbox{ $(a_1,,..,a_k,b,a_{k+1},...,a_n)$ is a Morley sequence in $\frak p$ over $A$ \ \ if and only if \ \  $\models\psi(b,\bar a)$ holds.}
\end{equation}
Let $\sigma_\frak p(x,a)$ be symmetric formula such that $\sigma_\frak p(x,a)$ relatively defines $\eps_\frak p(a)$.  We have: \ \ 
$$\bigcup_{1\leq i\leq n}p(y_i)\cup \{\sigma_\frak p(y_i,\Mon)<  y_{i+1} \,|\,1\leq i< n\,\} \cup \{\sigma_\frak p(y_k,\Mon)<b<\sigma_\frak p(y_{k+1},\Mon)\}\forces \psi(b,\bar y)\,.
$$
By compactness, there is $\phi(y)\in p(y)$ such that:
\begin{equation}\label{Eq_l2}\models(\forall \bar y)\left(\left(\bigwedge_{1\leq i\leq n}\phi(y_i)\wedge \bigwedge_{1\leq i< n}\sigma_\frak p(y_i,\Mon)<  y_{i+1} \wedge  \sigma_\frak p(y_k,\Mon)<b<\sigma_\frak p(y_{k+1},\Mon)\right)\Rightarrow  \psi(b,\bar y)\right)\,.
\end{equation}
Denote this formula by $\phi'(b)$; clearly $\phi'(x)\in p(x)$. Replace $<$ by $\phi'(\Mon)^2\cap <$\,. We {\em claim} that the modified $<$ witnesses that $\frak p$ is convex over $A$. Indeed, suppose that  $c,c'$ realize $p$ and that $c<b'<c'$; then, due to the modification of $<$, we have $\models\phi'(b')$. Choose $a_1',...,a_n'$ realizing $p$ such  that 
$$\eps_\frak p(a_1')<...<\eps_\frak p(a_k')<c <b'<c'<\eps_\frak p(a_{k+1}')<...<\eps_\frak p(a_n ')\,.$$
Clearly $a_1'...a_n'b'$ satisfy the left hand side of the implication in (\ref{Eq_l2}) so $\models \psi(b',\bar a')$ holds and, by (\ref{Eq_l1}), $(a_1,,..,a_k,b,a_{k+1},...,a_n)$ is a Morley sequence in $\frak p$ over $A$. In particular $b'\models p$. Thus $c<b'<c'$ implies that $b'$ realizes $p$, so $\frak p$ is convex over $A$. A contradiction.

\smallskip The remaining cases reduce to the first one. Suppose that $\psi(x,\bar a)\forces a_n<x$.  Then choose $a_{n+1}$ realizing $p$ such that $a_n<\eps_\frak p(a_{n+1})$. Define:
$$\psi'(x,\bar a,a_{n+1}):=\psi(x,\bar a) \wedge \sigma_\frak p(a_n,\Mon)<x<\sigma_\frak p(a_{n+1},\Mon)\, .$$ 
Then: \ \ $(a_1,...,a_n, b,a_{n+1})$ is a Morley sequence in $\frak p$ over $A$ \ \ if and only if \ \  $\models\psi'(b,\bar a,a_{n+1})$ holds. This reduces the second, and similarly the third, case to the proved one. The proof of the proposition and  Theorem \ref{T_omit_uvod} are now  completed. 
\end{proof}

Since there are continuum many non-isomorphic countable linear orders, we draw an immediate corollary.

\begin{cor} 
Suppose that $T$ is countable, $A$ is finite, and   $\mathfrak p$ is weakly regular and $A$-asymmetric. If $I(\aleph_0,T)<2^{\aleph_0}$ then $p$ is simple and convex over $A$.
\end{cor}

\section{Orthogonality}\label{X_orth}

In this  section we study orthogonality of regular types and prove Theorem \ref{T_asym_orth_uvod}. Orthogonality and  weak orthogonality  are essentially distinct relations on asymmetric regular types, as the following example shows.

\begin{exm}\label{E_wor_neq_perp}
Let $(\Mon,<)$ be a dense linear order without endpoints. Let $\frak p$ be the global type of an infinitely large element and let $\frak q$ be the type of an infinitely small element. Then $\frak p$ and $\frak q$ are $\emptyset$-invariant and  \ $\frak p\perp\frak q$. On the other hand \ $\frak p_{\strok \emptyset}= \frak q_{\strok \emptyset}$ \  and hence \ $\frak p_{\strok \emptyset}\nwor \frak q_{\strok \emptyset}$. 
\end{exm}

\begin{defn}\label{D_I(b)_D(b)}
Let $\mathfrak p$ be regular   and $A$-asymmetric.   For {\em any} tuple $\bar b$ define \ 
\begin{center}
 $D_{\mathfrak p,A}(\bar b):=\{a\models p\,|\,a\nmodels\mathfrak p_{\strok A\bar b}\}$ \ \  and \
\ $I_{\mathfrak p,A}(\bar b):=\mathfrak p_{\strok A\bar b}(\Mon)$.
\end{center}
\end{defn}

$I$ and $D$ stand for `independent' and `dependent' respectively, because realizations of   $\mathfrak p_{\strok A\bar b}$  are considered as being `independent from $\bar b$ over $A$'. Usually the meaning of $A$ will be clear from the context so we will write simply $I_\frak p(\bar b)$ and $D_\frak p(\bar b)$.

\begin{assu}Throughout this section we will assume that   $\frak p$ is regular   and $A$-asymmetric witnessed by $<$. By $p$ we will denote $\frak p_{\strok A}$. 
\end{assu}

\begin{lem}\label{Ld<i}
\begin{enumerate}[(i)]
\item $\tp(\bar b/A)\nwor p$ \ if and only if \ $D_{\mathfrak p}(\bar b)\neq \emptyset$.

\item If $D_{\mathfrak p}(\bar b)$ is nonempty, then it is a convex, downwards closed subset of $p(\Mon)$.

\item $D_{\mathfrak p}(\bar b)<I_{\mathfrak p}(\bar b)$ holds for any tuple $\bar b$ for which $D_{\mathfrak p}(\bar b)\neq\emptyset$. 
\end{enumerate}
\end{lem}
\begin{proof}(i) is easy.  To prove (ii) assume that $a'<a$ realize $p$ and that $a\in D_\frak p(\bar b)$. It suffices to show that $a'\in D_\frak p(\bar b)$. Otherwise, $a'\in I_\frak p(\bar b)$ combined with $a\nmodels \frak p_{\strok A\bar b}$, by regularity, implies $a'\models \frak p_{\strok A\bar ba}$ and thus $a<a'$. A contradiction.

\smallskip 
(iii) 
 Suppose that $a\in I_{\mathfrak p}(\bar b)$ and $c\in D_{\mathfrak p}(\bar b)$. Then $a\models\mathfrak p_{\strok A\bar b}$, $c\models\mathfrak p_{\strok A}$ and $c\nmodels \mathfrak p_{\strok A\bar b}$. The regularity condition implies $a\models \mathfrak p_{\strok A\bar bc}$ so, in particular, $c<a$.
\end{proof}

Consider the set of all tuples ordered by the inclusion of the corresponding $I_{\mathfrak p}(-)$'s. It is clearly  a quasi-order and, by the next lemma,  it is a total quasi-order: any pair of tuples is comparable.

\begin{lem}\label{Ltotal} For any pair  $\bar b$, $\bar c$ of  tuples (of possibly distinct size)  at least one of  $I_{\mathfrak p}(\bar b)\subseteq I_{\mathfrak p}(\bar c)$ and $I_{\mathfrak p}(\bar c)\subseteq I_{\mathfrak p}(\bar b)$ holds. Therefore, at least one of  $D_{\mathfrak p}(\bar c)\subseteq D_{\mathfrak p}(\bar b)$ and $D_{\mathfrak p}(\bar b)\subseteq D_{\mathfrak p}(\bar c)$ holds, too.
\end{lem}
\begin{proof} Suppose for a contradiction that neither of them holds and choose $a\in I_\frak p(\bar b)\smallsetminus I_\frak p(\bar c)$ and $d\in I_\frak p(\bar c)\smallsetminus I_\frak p(\bar b)$. By regularity, 
$a\models \mathfrak p_{\strok A\bar b}$ and $d\nmodels \mathfrak p_{\strok A\bar b}$ imply $a\models \mathfrak p_{\strok A\bar bd}$; in particular $d<a$. Similarly, $d\models \mathfrak p_{\strok A\bar c}$ and $a\nmodels \mathfrak p_{\strok A\bar c}$ imply $d\models \mathfrak p_{\strok A\bar ca}$ and $a<d$. A contradiction.  
\end{proof}

\begin{lem}\label{LDscl_closed}
$D_{\mathfrak p}(\bar c)$ is $\eps_\frak p$-closed for any tuple $\bar c$:     $a\in  D_{\mathfrak p}(\bar c)$ \  implies \ $\eps_\frak p(a)\subseteq D_{\mathfrak p}(\bar c)$. 
\end{lem}
\begin{proof}Suppose that $D_{\mathfrak p}(\bar c)$ is not $\eps_\frak p$-closed and work for a contradiction. Let $a\in D_{\mathfrak p}(\bar c)$ and $b\in\eps_\frak p(a)$ be such that $b\in I_{\mathfrak p}(\bar c)$.  Let $b'$ realize $\mathfrak p_{\strok A\bar cab}$. Then \ \ $D_{\mathfrak p}(\bar c)<b<\eps_\frak p(b')\ .$ \
The first inequality holds by Lemma \ref{Ld<i}(iii), and the second follows from $b'\models\mathfrak p_{\strok Ab}$.
Both $b$ and $b'$ belong to $I_{\mathfrak p}(\bar c)$; in particular, they have the same type over $A\bar c$ so there is  $f\in \Aut_{A\bar c}(\Mon)$ such that $f(b)=b'$. Let $a'=f(a)$. Since $a\in\eps_\frak p(b)$,   $a,b\equiv a',b'\,(A\bar c)$ implies   $a'\in\eps_\frak p(b')$ and by the above $D_{\mathfrak p}(\bar c)<a'$ holds. Since   $D_{\mathfrak p}(\bar c)$ is fixed setwise by $f$ and it contains $a$, then $a'=f(a)$ implies  $a'\in D_{\mathfrak p}(\bar c)$. A contradiction.    
\end{proof}

 In the following proposition we prove that asymmetric regular types intuitively  `have weight one with respect to $\nwor$ '.  

\begin{prop}\label{prop_nwor_transitivity} If  $q,r\in S(A)$ are such that $p\nwor q$ and $p\nwor r$ then $q\nwor r$.
\end{prop}
\begin{proof} We will show that both \ $q(\bar y)\cup r(\bar z)\cup \{D_\frak p(\bar y)\subsetneq D_\frak p(\bar z)\}$ \ and \ $q(\bar y)\cup r(\bar z)\cup \{D_\frak p(\bar z)\subsetneq D_\frak p(\bar y)\}$ \ are consistent (here $D_\frak p(\bar y)\subsetneq D_\frak p(\bar z)$ is expressed by an $\mathcal L_{\infty,\omega}$-formula);   it clearly follows that   
$q(\bar y)\cup r(\bar z)$ has at least two distinct completions and $q\nwor r$. Let $\bar b\models q$. Then, by Lemma \ref{Ld<i}(i), $D_\frak p(\bar b)\neq\emptyset$. Choose $a\in D_\frak p(\bar b)$ and $a'\in I_\frak p(\bar b)$. By Lemma \ref{Ld<i}(iii) we derive 
$D_\frak p(\bar b)<a'$. Now $p\nwor r$ implies that $D_\frak p(\bar c)\neq\emptyset$ holds for any $\bar c\models r$; choose such a $\bar c$ that $a'\in D_\frak p(\bar c)$. Then $a'\in D_\frak p(\bar c)\smallsetminus D_\frak p(\bar b)$ so, by Lemma \ref{Ltotal}, $D_\frak p(\bar b)\subsetneq D_\frak p(\bar c)$. Thus $(\bar b,\bar c)$ satisfies $q(\bar y)\cup r(\bar z)\cup \{D_\frak p(\bar y)\subsetneq D_\frak p(\bar z)\}$. A similar argument shows that $q(\bar y)\cup r(\bar z)\cup \{D_\frak p(\bar z)\subsetneq D_\frak p(\bar y)\}$ is consistent. 
\end{proof}

\begin{lem}\label{LDworsas}
Suppose that     $\mathfrak q$ is $A$-invariant  and $p\nwor\mathfrak q_{\strok A}$. Let $(\bar b,\bar b')\models (\mathfrak q^2)_{\strok A}$. 
\begin{enumerate}[(i)] 
\item   \ $(\mathfrak p\otimes\mathfrak q)_{\strok A}= (\mathfrak q\otimes \mathfrak p)_{\strok A} $ \   \  implies \ \  $D_\frak p(\bar b')\subsetneq D_\frak p(\bar b)$.

\item \ $(\mathfrak p\otimes\mathfrak q)_{\strok A}\neq (\mathfrak q\otimes \mathfrak p)_{\strok A}$ \     \ implies \ \  $D_\frak p(\bar b)\subsetneq D_\frak p(\bar b')$. 

\item  \  $\mathfrak q$ is asymmetric.
\end{enumerate}
\end{lem}
\begin{proof}Fix $\bar b\models \frak q_{\strok A}$. Then, by Lemma \ref{Ld<i}(i), $p\nwor\mathfrak q_{\strok A}$ implies    $D_\frak p(\bar b)\neq\emptyset$. 

\smallskip (i) Choose   $a\in D_\frak p(\bar b)$ such that  $\bar b'\models \mathfrak q_{\strok A\bar ba}$. Then $\bar b'\models \mathfrak q_{\strok A a}$ so, by commutativity, $a\models \mathfrak p_{\strok A\bar b'}$; in particular, $a\notin D_\frak p(\bar b')$. Hence $a\in D_\frak p(\bar b)\smallsetminus D_\frak p(\bar b')$ and, by Lemma \ref{Ltotal}, we derive $D_\frak p(\bar b')\subsetneq D_\frak p(\bar b)$. Since the inclusion does not depend on the particular choice of $a$, it holds for all $(\bar b,\bar b')\models (\mathfrak q^2)_{\strok A}$.

\smallskip (ii) Choose $a$ such that  $(\bar b,a, \bar b')\models (\mathfrak q\otimes\mathfrak p\otimes\mathfrak q)_{\strok A}$. By Lemma \ref{Ld<i} $a\in I_\frak p(\bar b)$ implies $D_\frak p(\bar b)<a$. $(\mathfrak p\otimes\mathfrak q)_{\strok A}\neq (\mathfrak q\otimes \mathfrak p)_{\strok A} $ and $\bar b'\models\mathfrak q_{\strok Aa}$ together imply $a\nmodels \mathfrak p_{\strok A\bar b'}$. Hence $a\in D_\frak p(\bar b')$. By combining that with $D_\frak p(\bar b)<a$ and applying Lemma \ref{Ltotal} we deduce $D_\frak p(\bar b)\subsetneq D_\frak p(\bar b')$. 

\smallskip (iii) By (i) and (ii) $\tp(b,b'/A)\neq\tp(b',b/A)$, so $q$ is asymmetric. 
\end{proof}

Now we can prove a version of Theorem \ref{T_asym_orth_uvod}.

\begin{thm}\begin{enumerate}[(i)]
\item   A regular asymmetric type is orthogonal to any invariant symmetric type.
\item  Both symmetry and asymmetry are preserved under non-orthogonality of regular types. 
\item $\nor$ is an equivalence relation on the set of asymmetric, regular types.
\end{enumerate} 
\end{thm}
\begin{proof}
(i) is direct consequence of Lemma \ref{LDworsas}. (ii) follows from (i), and (iii) follows directly from Proposition \ref{prop_nwor_transitivity}.
\end{proof}

\section{Regularity and $\nwor$}\label{X_strongreg}

In this section we study preservation of invariants of asymmetric regular types under $\nwor$. In general there may be no connection between them as the following  example shows.

\begin{exm}\label{Ex_nonconvex_1} 
Consider the structure $(\Mon,<,P_i)_{i\in\omega}$ from Example  \ref{Ex_nonconvex}.  $\mathfrak p$ is the type of an infinitely large element satisfying  $\{\neg P_i(x)\,|\,i\in\omega\}$, and $\mathfrak p_n$ is the type of an infinitely large element satisfying $P_n(x)$. Then $\frak p\nor \frak p_n$ because  $\frak p(x)\cup\frak p_n(y)$ is consistent with either of $x<y$ and $y<x$;  from the same reason $\frak p_{\strok  A}\nwor \frak p_n\,_{\strok A}$ holds. There are arbitrarily large models omitting $\frak p_{\strok\emptyset}(x)$. More precisely,  $M\smallsetminus\frak p_{\strok\emptyset}(M)\prec M$ holds for all $M$. Therefore  $\Inv_{\frak p_n,\emptyset}(M)$ can be an arbitrary dense linear order without end points, while $\Inv_{\frak p,\emptyset}(M)$ is empty.  
\end{exm}

The situation is much better when we assume that the types are convex.

\begin{assu}\label{ass_sr} \ Throughout the section we assume:
\begin{enumerate}
\item $\frak p(x)$  and $\frak q(y)$  are regular, $A$-asymmetric and convex over $A$.
\item Both the $A$-asymmetry and convexity are witnessed by $<_{\mathfrak p}$ and $<_{\mathfrak q}$ respectively.
\item $p\nwor q$ \ \ where $p$ and $q$ denote the corresponding restrictions of $\frak p$ and $\frak q$ to $A$.
\end{enumerate}
\end{assu}

\noindent We will consider two essentially distinct kinds of non-orthogonality:

\medskip {\bf Bounded} \ \ There are $a\models p$, $b\models q$ and  $\theta(x,y)\in \tp(a,b/A)$ such that $\theta(a,\Mon)$ is $\frak q$-bounded. 

\medskip {\bf Unbounded} \ \ Otherwise.

\smallskip\noindent 
These two types will be considered in separate subsections where the corresponding parts of Theorem \ref{T_nor_uvod} are proved.

\subsection{The bounded case} \label{X_bdd_nwor}

\begin{exm}\label{E_bdd_nwor} \ Examples of  bounded $\nwor$.

(1) \ Consider    $(\mathbb Q\times\{0,1\}, P(M),Q(M),<_\frak p,<_\frak q, S)$ where $(P(M),<_\frak p)$ and $(Q(M),<_\frak q)$ are the naturally ordered $\mathbb Q\times \{0\}$ and $\mathbb Q\times \{1\}$ respectively, and $S((a,0),(b,1))$ holds iff $a=b$.   $\frak p$  and $\frak q$  are types of infinitely large elements containing $P(x)$ and $Q(x)$ respectively.   $\frak p, \frak q$ are regular,   $\emptyset$-asymmetric, and $\frak p_{\strok \emptyset}\nwor\frak q_{\strok \emptyset}$. Morley sequences are increasing sequences and (in any model) the isomorphism between   $\Inv_\frak p(M)$ and $\Inv_\frak q(M)$ is determined by $S$ viewed as a function. Note that $\frak p\otimes\frak q\neq \frak q\otimes\frak p$.

\smallskip 
(2) \ Consider  $(M, P(M),Q(M),<_\frak p,<_\frak q, S)$ where   $(P(M),<_\frak p)$ is the lexicographically ordered $\omega_1\times \mathbb Q$ and $(Q(M),<_\frak q)$ is the lexicographically ordered $\omega_1^*\times \mathbb Q$ (here $\omega_1^*=\{\alpha^*\,|\,\alpha\in\omega_1\}$  is the reversely ordered $\omega_1$). $M$ is the disjoint union of $P(M)$ and $Q(M)$, and $S((\alpha,r),(\beta^*,s))$ holds if and only if $\alpha= \beta$ and $r= s$. Let $\frak p$ be the type of an infinitely $<_\frak p$-large element of $P(M)$ and let $\frak q$ be the type of an infinitely $<_\frak q$-large element of $Q(M)$. Then Morley sequences in $\frak p$ ($\frak q$) over $\emptyset$ are $<_\frak p$-increasing ($<_\frak q$-increasing). $\Inv_\frak p(M)$ is isomorphic to $(\omega_1\times \mathbb Q ,<_\frak p)$, $\Inv_\frak q(M)$ is isomorphic to $(\omega_1^*\times \mathbb Q,<_\frak q)$; they are not isomorphic, but they are anti-isomorphic by $S$.        
 Note that $\frak p\otimes\frak q= \frak q\otimes\frak p$. 
 
 \smallskip (3) $\frak p$ and $\frak q$ in both (1) and (2)  are simple over $\emptyset$. Non-simple examples can be easily made by `multiplying lexicographically' each structure   by $(\mathbb Z,<)$, and requiring that $S((a,m),(b,n)$ holds iff $S(a,b)$ holds in the original structure. Here $\eps_\frak p((a,n))$ is the copy of $\mathbb Z$ around $a$ and   $S$ induces  a function mapping copies of $\mathbb Z$ around realizations of $\frak p_{\strok A}$ to copies of $\mathbb Z$ around realizations of $\frak q_{\strok A}$; in other words, it maps $\Lin_A(\frak p)$ to $\Lin_A(frak q)$ and is an (anti) isomorphism of the corresponding linear orders.   
\end{exm}

Throughout the subsection   assume that   $p\nwor q$ is a bounded type of non-orthogonality: there are $a\models p$, $b\models q$ and  $\theta(x,y)\in \tp(a,b/A)$ such that $\theta(a,\Mon)$ is $\frak q$-bounded, i.e. $\pi_\frak q(\theta(a,\Mon)\cap q(\Mon))$ is bounded in $\Lin_A(\frak q)$. We will show that the situation from the previous examples holds: $\Inv_A(\frak p)$ is either isomorphic or anti-isomorphic to $\Inv_A(\frak q)$, and isomorphism (anti-isomorphism) is naturally induced by $\theta$.

\begin{lem}\label{Lbounded_unbounded}
Suppose that $a$ realizes $p$, $b$ realizes $q$,  $\theta(x,y)\in\tp(a,b/A)$, and that $\theta(a,y)$ is $\frak q$-bounded over $A$. Then:
\begin{enumerate}[(i)]
\item If $(a,a')$ is a Morley sequence in $\frak p$ over $A$ then $\pi_{\frak q}(\theta(a,\Mon)\cap q(\Mon))$  and $\pi_{\frak q}(\theta(a',\Mon)\cap q(\Mon))$ \ are disjoint.

\item  $ \pi_{\frak p}(\theta(\Mon,b)\cap p(\Mon))=\{\eps_\frak p(a)\}$; \ in particular, $\theta(x,b)$ is  $\frak p$-bounded. 

\item $\theta(a,\Mon)\cap q(\Mon)\subseteq \eps_\frak p(b)$ and $\theta(\Mon,b)\cap p(\Mon)\subseteq \eps_\frak p(a)$.

\item $a\in \Sem_A(b)$ and $b\in\Sem_A(a)$. 
\end{enumerate}  
\end{lem}
\begin{proof}
(i) To simplify notation denote $\pi_{\frak q}(\theta(a,\Mon)\cap q(\Mon))$ simply by $\pi_{\frak q}(a)$. Suppose that the conclusion fails. Then whenever $(a,a')$ is a Morley sequence    $\pi_{\frak q}(a)\cap\pi_{\frak q}(a')\neq\emptyset$ holds. Let $\pi^c_\frak p(x)$ denote the convex closure of $\pi_\frak p(x)$ in $\Lin_A(\frak p)$. Then: whenever $(a,a')$ is a Morley sequence    $\pi_{\frak q}(a)\cap\pi_{\frak q}(a')\neq\emptyset$ holds.

Since  $\pi^c_{\frak q}(-)$ is bounded for all realizations of $p$, there are  $a_1,a_2,a_3$ realizing $p$ such that $\pi^c_{\frak q}(a_1)<_{\frak q}\pi^c_{\frak q}(a_2)<_{\frak q}\pi^c_{\frak q}(a_3)$. 
Choose $a_4$ realizing $\frak p_{\strok Aa_1a_2a_3}$. Then  $(a_i,a_4)$ is a Morley sequence for $i=1,2,3$ so   $\pi^c_{\frak q}(a_4)$ meets 
each   $\pi^c_{\frak q}(a_i)$. Since all of them are convex the middle one has to be fully contained in $\pi^c_{\frak q}(a_4)$:   $\pi^c_{\frak q}(a_2)\subsetneq \pi^c_{\frak q}(a_4)$.  That   inclusion holds for any pair realizing a Morley sequence in place of $(a_2,a_4)$. Now take $a_0$ such that $\eps_\frak p(a_0)<_{\frak p}\eps_\frak p(a_1)\cup\eps_\frak p(a_2)$. Then $(a_0,a_1)$ and $(a_0,a_2)$ are Morley sequences, so $\pi^c_{\frak q}(a_0)$ is contained in each of $\pi^c_{\frak q}(a_1)$ and $\pi^c_{\frak q}(a_2)$, so $\pi^c_{\frak q}(a_1)\cap\pi^c_{\frak q}(a_2)\neq\emptyset$.\ A contradiction.

\smallskip 
(ii) By part (i) $\theta(\Mon,b)$ does not contain a pair realizing  a Morley sequence in $\frak p$ over $A$. Hence any element of $\theta(\Mon,b)\cap p(\Mon)$ is in $\eps_\frak p(a)$ so  $ \pi^c_{\frak p}(\theta(\Mon,b)\cap p(\Mon))=\{\eps_\frak p(a)\}$. Clearly, $\theta(x,\bar b)$ is $\frak p$-bounded.  

\smallskip 
(iii) Follows immediately from part (ii) applied to $\theta(a,y)$ and $\theta(x,b)$. 

\smallskip
(iv) Since $\theta(a,y)$ is $\frak q$-bounded, $\theta(x,b)$ is $\frak p$-bounded and $\frak p$ and $\frak q$ are convex, by Lemma \ref{L_p_bounded_fla}(ii) we can choose $\varphi_\frak p(x)\in p(x)$ and $\varphi_\frak q(y)\in q(y)$ such that $\varphi_\frak p(x)\wedge\theta(x,b)$ witnesses $a\in\Sem_A(b)$ and $\varphi_\frak q(y)\wedge\theta(a,y)$ witnesses $b\in\Sem_A(a)$.
\end{proof} 
 
We say that a formula $\theta(x,y)$ is {\em $(\frak p,\frak q)$-bounded (over $A$)} if it is consistent with $p(x)\cup q(y)$, $\theta(a,\Mon)\subseteq\eps_\frak q(b)$,  and $\theta(\Mon,b)\subseteq\eps_\frak p(a)$ for some (any) $a\models p$ and $b\models q$ satisfying $\models \theta(a,b)$.    

\begin{rmk}\label{rmk_both_bounded}  
By Lemma \ref{Lbounded_unbounded}, for any formula $\theta(x,y)$ consistent with $p(x)\cup q(y)$, the following conditions are equivalent:
\begin{enumerate}
\item $\theta(a,y)$ is $\frak q$-bounded for some $a\models p$;
\item $\theta(x,b)$ is $\frak p$-bounded for some $b\models q$;
\item There are   $\varphi_{\frak p}(x)\in p(x)$ and   $\varphi_{\frak q}(y)\in q(y)$ such that   $\varphi_{\frak p}(x)\wedge\varphi_{\frak q}(y)\wedge\theta(x,y)$ is $(\frak p,\frak q)$-bounded. 
\end{enumerate} 
The equivalence of the first two follows from part (iii) of the lemma,  and the equivalence of the third with them follows from (the proof of) (iv).     
\end{rmk}

\begin{lem}\label{LFG_definition}
Suppose that  $\theta(x,y)\in\tp(a,b/A)$ is   $(\frak p,\frak q)$-bounded over $A$. Then for all $a,a'$ realizing $\frak p$: \ 
  $a'\in \eps_{\frak p}(a)$ \ if and only if \  $\pi_{\frak q}(\theta(a,\Mon)\cap q(\Mon))=\pi_{\frak q}(\theta(a',\Mon)\cap q(\Mon))$.
\end{lem}
\begin{proof} To simplify notation, for $a\models p$ and $b\models q$, we write $\pi_\frak p(b)$ and $\pi_\frak q(a)$ instead of $\pi_{\frak p}(\theta(\Mon,b)\cap p(\Mon))$ and $\pi_{\frak q}(\theta(a,\Mon)\cap q(\Mon))$, respectively.

$\Leftarrow)$ \ Suppose that $a'\notin \eps_{\frak p}(a)$ and $\pi_{\frak q}(a)= \pi_{\frak q}(a')=\{\eps_{\frak q}(b)\}$ and work for a contradiction. Without loss of generality, assume  that $\eps_{\frak p}(a)<_\frak p\eps_{\frak p}(a')$, i.e. $(a,a')$ is a Morley sequence in $\frak p$ over $A$. Then  $\pi_\frak q(c)=\pi_\frak q(c')$ holds for every Morley sequence $(c,c')$ in $\frak p$ over $A$. Choose $b_1$ such that $\eps_{\frak q}(b)\neq\eps_{\frak q}(b_1)$ and let $f\in\Aut_A(\Mon)$ which maps $(a,a',b)$ to $(a_1,a_1',b_1)$. Then $\pi_{\frak q}(a_1)=\pi_{\frak q}(a_1')=\eps_{\frak q}(b_1)$. Since both $(a,a')$ and $(a_1,a_1')$ are Morley sequences, at least one of $(a,a_1')$ and $(a_1,a')$ is a Morley sequence; assume that $(a,a_1')$ is. Then we have  $\pi_{\frak q}(a)=\pi_{\frak q}(a_1')$, i.e. $\eps_{\frak q}(b)=\eps_{\frak q}(b_1)$. A contradiction. In a similar way we get a contradiction in the case when $(a_1,a')$ is a Morley sequence.

\smallskip 
$\Rightarrow)$ \ Assume on the contrary that $a'\in\eps_\frak p(a)$ and $\pi_\frak q(a)= \{\eps_\frak q(b)\}$, $\pi_\frak q(a')= \{\eps_\frak q(b')\}$ and $\eps_\frak q(b)\neq \eps_\frak q(b')$. Then $b'\notin\eps_\frak q(b)$, hence, similarly as in the proof of $\Rightarrow)$, $\pi_\frak p(b)\neq \pi_\frak p(b')$. Since $\pi_\frak p(b)= \{\eps_\frak p(a)\}$ and $\pi_\frak p(b')= \{\eps_\frak p(a')\}$, we get $\eps_\frak p(a)\neq \eps_\frak p(a')$, and therefore $a'\notin \eps_\frak p(a)$, which is a contradiction.
\end{proof}

An immediate consequence of Lemma \ref{LFG_definition} is that for any $(\frak p,\frak q)$-bounded formula $\theta(x,y)$  
$$F_{\theta}(\eps_\frak p(a))=\eps_\frak q(b)  \ \ \mbox{iff} \ \ \models \theta(a,b)\ \ \mbox{iff} \ \ G_{\theta}(\eps_\frak q(b))=\eps_\frak p(a)$$
properly defines functions \ $F_{\theta}:\Lin_A(\frak p)\longrightarrow \Lin_A(\frak q)$ and $G_{\theta}:\Lin_A(\frak q)\longrightarrow \Lin_A(\frak p)$. \ Since $F_{\theta}\circ G_{\theta}$ and $G_{\theta}\circ F_{\theta}$ \ are identity maps they are bijections. In the next lemma we show that they do not depend on the choice of $\theta$.

\begin{lem}\label{L_pq_teta_teta'}
If   $\theta(x,y)$ and $\theta'(x,y)$ are   $(\frak p,\frak q)$-bounded over $A$ then  $F_{\theta}=F_{\theta'}$ and $G_{\theta}=G_{\theta'}$. 
\end{lem}
\begin{proof}Towards  contradiction assume that   $a,b,b'$ are such that $F_{\theta}(\eps_{\frak p}(a))=\eps_{\frak q}(b)\neq\eps_{\frak q}(b')=F_{\theta'}(\eps_\frak p(a))$;  without loss of generality assume $\eps_{\frak q}(b)<_\frak q\eps_{\frak q}(b')$. Since both $\theta(x,y)$ and $\theta'(x,y)$ are $(\frak p,\frak q)$-bounded we have $\theta(a,\Mon)\subseteq\eps_{\frak q}(b)$ and $\theta'(a,\Mon)\subseteq\eps_{\frak q}(b')$, so $\theta(a,\Mon)<_{\frak q}\theta'(a,\Mon)$. Let $\psi(a,y)$ be the formula $\theta(a,\Mon)<_\frak q y<_\frak q \theta'(a,\Mon)$. Then  $b<_{\frak q}\psi(a,\Mon)<_{\frak q}b'$, which implies that $\psi(a,y)$ is $\frak q$-bounded. On the other hand $\psi(a,y)$ is satisfied by all elements satisfying $b<_{\frak q}\eps_\frak q(y)<_{\frak q}b'$, which is not possible by Lemma \ref{Lbounded_unbounded}(ii).
\end{proof}

The lemma implies that $F_{\theta}$ and $G_{\theta}$ are canonical maps not depending on the particular choice of the $(\frak p,\frak q)$-bounded formula $\theta$. From now on we will simply denote them by $F$ and $G$.
 
\begin{lem}\label{L_F(a)_max_in_D(a)} Suppose that $a\models p$, $b\models q$. Then:
\begin{enumerate}[(i)]
\item \ $b'\models \frak q_{\strok Aa}$  \ if and only if  \ $F(\eps_\frak p(a))<_\frak q \eps_\frak q(b')$;
\item \ $a'\models \frak p_{\strok Ab}$  \ if and only if  \ $G(\eps_\frak q(b))<_\frak p \eps_\frak p(a')$.
\end{enumerate}
\end{lem}
\begin{proof}
We prove (i), the part (ii) is proved by a similar argument. Let $F(\eps_\frak p(a))=\eps_\frak q(b)$ and let $\theta(x,y)$ be a $(\frak p,\frak q)$-bounded formula.  To prove $\Rightarrow$) assume that $b'\models \frak q_{\strok Aa}$. Note that the formula `saying '$\theta(a,\Mon)<_\frak q y$' has no upper $<_\frak q$-bound, so it belongs to $\frak q_{\strok Aa}$.  Thus $ \frak q_{\strok Aa}(y)\vdash b<y$. Similarly, by replacing $\theta(a,y)$ by $\exists y'(\theta(a,y)\wedge \sigma(y',y))$  for $\sigma(y',y)\in \mathcal B_c$,   we conclude 
that $\frak q_{\strok Aa}(y)\vdash \eps_\frak q(b)<y$. Therefore  \ $F(\eps_\frak p(a))<_\frak q \eps_\frak q(b')$, completing the proof of $\Rightarrow$). 

Assume that $\Leftarrow$) fails to be true and work for a contradiction. Let $b'\models q$ be such that  $b'\nmodels \frak q_{\strok Aa}$ and  $F(\eps_\frak p(a))<_\frak q \eps_\frak q(b')$; then   $\eps_\frak q(b))<_\frak q \eps_\frak q(b')$.  
 $b'\nmodels\frak q_{\strok Aa}$ implies that there exists $\theta'(a,y)\in\tp(b/Aa)$  such that $\theta(a,\Mon)\cap q(\Mon)$ has an upper $<_\frak q$-bound. Let 
 $\theta''(a,y)$ be \ $\theta(a,\Mon)<_\frak q y\wedge \theta'(a,y)$.\ Clearly,  $b'$ satisfies $\theta''(a,y)$. \ $\theta''(a,y)$ is $<_\frak q$-bounded from above because $\theta'(a,y)$ is so; it is  $<_\frak q$-bounded from below  because any solution is bigger than $b\in\theta(a,\Mon)$. Hence $\theta''(a,y)$ is $\frak q$-bounded. By Remark \ref{rmk_both_bounded} we may slightly modify it so that it is $(\frak p,\frak q)$-bounded.   By Lemma \ref{L_pq_teta_teta'} $F_{\theta''}(\eps_\frak p(a))= F_{\theta}(\eps_\frak p(a))$; but $F_{\theta''}(\eps_\frak p(a))= \eps_\frak q(b')$ and $F_\theta(\eps_\frak p(a))=\eps_\frak q(b)$. A contradiction.
\end{proof}

\begin{lem}\label{L_boundedlemma_1}   
\begin{enumerate}[(i)]  \item   $(\mathfrak p\otimes \mathfrak q)_{\strok A}\neq (\mathfrak q\otimes \mathfrak p)_{\strok A}$ \ if and only if \  $F$ is increasing.  

\item   $(\mathfrak p\otimes \mathfrak q)_{\strok A}= (\mathfrak q\otimes \mathfrak p)_{\strok A}$ \ if and only if \  $F$ is  decreasing.
\end{enumerate}
\end{lem}
\begin{proof} Suppose that $(a,a')\models \frak p^2_{\strok A}$ and let $b,b'$ be such that $F(\eps_\frak p(a))= \eps_\frak q(b)$ and $F(\eps_\frak p(a'))= \eps_\frak q(b')$. Clearly, $\eps_\frak q(b)\neq\eps_\frak q(b')$ so we have two possible cases:

\smallskip {\bf Case 1.} \ \ $\eps_\frak q(b)<_\frak q\eps_\frak q(b')$ \ \ (i.e. $F(\eps_\frak p(a))<_\frak q F(\eps_\frak p(a'))$).

By Lemma \ref{L_F(a)_max_in_D(a)}(i) we have  $(a,b')\models (\mathfrak p\otimes \mathfrak q)_{\strok A}$, and by Lemma \ref{L_F(a)_max_in_D(a)}(ii) $(b',a)\nmodels (\mathfrak q\otimes \mathfrak p)_{\strok A}$. Therefore $(\mathfrak p\otimes \mathfrak q)_{\strok A}\neq(\mathfrak q\otimes \mathfrak p)_{\strok A}$.

\smallskip {\bf Case 2.} \ \ $\eps_\frak q(b')<_\frak q\eps_\frak q(b)$ \ \ (i.e. $F(\eps_\frak p(a'))<_\frak q F(\eps_\frak p(a))$).

Similarly, by Lemma \ref{L_F(a)_max_in_D(a)}(i) we have  $(a',b)\models (\mathfrak p\otimes \mathfrak q)_{\strok A}$, and by Lemma \ref{L_F(a)_max_in_D(a)}(ii) $(b,a')\models (\mathfrak q\otimes \mathfrak p)_{\strok A}$. Therefore $(\mathfrak p\otimes \mathfrak q)_{\strok A}=(\mathfrak q\otimes \mathfrak p)_{\strok A}$. 

\smallskip 
Now the proof of the lemma follows easily.
\end{proof}

\begin{cor}
\begin{enumerate}[(i)]
\item If $(\mathfrak p\otimes \mathfrak q)_{\strok A}\neq(\mathfrak q\otimes \mathfrak p)_{\strok A}$ then $F$ is an isomorphism of $(\Lin_A(\frak p),<_\frak p)$ and $(\Lin_A(\frak q),<_\frak q)$.

\item If $(\mathfrak p\otimes \mathfrak q)_{\strok A}=(\mathfrak q\otimes \mathfrak p)_{\strok A}$ then $F$ is an isomorphism of $(\Lin_A(\frak p),<_\frak p)$ and $(\Lin_A(\frak q),>_\frak q)$.
\end{enumerate}
\end{cor}
 
\smallskip 
Recall from Section 2 that we denoted $(\Lin_A(\frak p),<_\frak p)$ by $\mathbb L_A(\frak p)$. We have just proved that $F$ is an isomorphism    of   $\mathbb L_A(\frak p)$ and   either $\mathbb L_A(\frak q)$  in the non-commutative case, or $\mathbb L_A^*(\frak q)$ in the commutative case. This isomorphism is canonical in the sense that it induces a bijection between $\Lin_{A,\frak p}(M)= \{\eps_\frak p^A(a)\cap M\mid a\in \frak p_{\strok A}(M)\}$ and $\Lin_{A,\frak q}(M)$ for any small model $M\supseteq A$. So for any $M\supseteq A$ by $F_M$ we denote the restriction of $F$ to $\Lin_{A,\frak p}(M)$. We will now prove Theorem \ref{T_nor_uvod}(i).   
 
\begin{prop}\label{P1} Suppose  $M\supseteq A$. 

\begin{enumerate}[(i)]  \item If $(\mathfrak p\otimes \mathfrak q)_{\strok A}\neq(\mathfrak q\otimes \mathfrak p)_{\strok A}$  \ then \ $F_M$ is an isomorphism of $(\Lin_{A,\frak p}(M),<_\frak p)$ and    $(\Lin_{A,\frak q}(M),<_\frak q)$. In  particular,  $\Inv_{\mathfrak p,A}(M)=\Inv_{\mathfrak q,A}(M)$. 

\item   If $(\mathfrak p\otimes \mathfrak q)_{\strok A}=(\mathfrak q\otimes \mathfrak p)_{\strok A}$ \ then \ $F_M$ is an isomorphism of $(\Lin_{A,\frak p}(M),<_\frak p)$ and $(\Lin_{A,\frak q}(M),>_\frak q)$. In particular,  $\Inv_{\mathfrak p,A}(M)$ is isomorphic to the reversely ordered $\Inv_{\mathfrak q,A}(M)$.
\end{enumerate}
\end{prop}
\begin{proof}
It suffices to show that $F_M$ maps $\Lin_{A,\frak p}(M)$ onto $\Lin_{A,\frak q}(M)$; then it follows that it is an isomorphism of the corresponding orders. So let $\theta(x,y)$ be a $(\frak p,\frak q)$-bounded formula and let $a\in p(M)$. By Lemma \ref{Lbounded_unbounded}(iv) we may assume that $\theta(a,y)\vdash q(y)$ so  the consistency of $\theta(a,y)$ implies that there is $b\in M$ satisfying $\theta(a,y)$. Therefore $F_M$ maps $\Lin_{A,\frak p}(M)$ onto $\Lin_{A,\frak q}(M)$.
\end{proof}

\subsection{Unbounded $\nwor$}\label{X_unbdd_nwor}

\begin{exm}\label{E_unbdd_nwor} \ Examples of unbounded $\nwor$.

(1) \ Consider the structure $(\mathbb R, P(\mathbb R),Q(\mathbb R),<_u,<_v, S)$ where $P,Q$ are unary, $Q(\mathbb R)=\mathbb Q$ and $P(\mathbb R)$ is the set of all irrationals; $<_\frak p$ and $<_\frak q$ are the corresponding restrictions of the natural ordering to $P(\mathbb R)$ and $Q(\mathbb R)$ respectively; $S(a,b)$ holds iff $a$ is rational, $b$ is irrational and $b<a$ (in the natural ordering). Let $\frak p$  and $\frak q$ respectively  be the global types of an infinitely large element satisfying $P(x)$  and $Q(x)$ respectively. It is straightforward to check that $\frak p, \frak q$ are regular,   $\emptyset$-asymmetric, and  that $\frak p_{\strok \emptyset}\nwor\frak q_{\strok \emptyset}$.  Then $\Inv_\frak p$ of our structure is the order type of irrationals, while its $\Inv_\frak q$ is the order type of rationales. However the  Dedekind completions of these invariants are isomorphic, and the isomorphism is induced by $S$. Note that $\frak p\otimes\frak q\neq\frak q\otimes\frak p$. 

\smallskip 
(2) \ Consider the structure $(M, P(M),Q(M),<_\frak p,<_\frak q, S^*)$ where $(P(M),<_\frak p)$ is the lexicographically ordered $\omega_1\times \mathbb Q$ and $(Q(M),<_\frak q)$ is the lexicographically ordered $\omega_1^*\times \mathbb (\mathbb R\smallsetminus\mathbb Q)$. $M$ is the disjoint union of $P(M)$ and $Q(M)$, and $S((\alpha,r),(\beta^*,s))$ holds if and only if $\alpha\leq \beta$ and $r< s$. Let $\frak p$ be the type of an infinitely $<_\frak p$-large element of $P(M)$ and let $\frak q$ be the type of an infinitely $<_\frak q$-large element of $Q(M)$. Morley sequences in $\frak p$ ($\frak q$) over $\emptyset$ are $<_\frak p$-increasing ($<_\frak q$-increasing). $\Inv_\frak p(M)$ and $\Inv_\frak q(M)$ are isomorphic to the lexicographically ordered $\omega_1\times \mathbb Q$ and  $\omega_1^*\times \mathbb (\mathbb R\smallsetminus\mathbb Q)$ respectively. Their Dedekind completions are not isomorphic but they are anti-isomorphic, and the anti-isomorphism is induced by $S^*$. Note that $\frak p\otimes\frak q= \frak q\otimes\frak p$
\end{exm}
 
If we assume that the types are strongly regular, we will show that the situation from preceding examples holds: Dedekind completion of $\Inv_A(\frak p)$ is either isomorphic or anti-isomorphic to Dedekind completion of $\Inv_A(\frak q)$, and the isomorphism (anti-isomorphism) is induced by $\theta$.
 
\begin{assu} \ In addition to Assumption \ref{ass_sr} we will assume  that $(\frak p(x),\phi_\frak p(x))$  and $(\frak q(y),\phi_\frak q(y))$  are strongly regular.
\end{assu}

Assume that we are not in the bounded case. Then for all $a\models p$, $b\models q$ and all $\theta(x,y)\in\mathrm{tp}(a,b/A)$, neither $\theta(x,b)$ is $\frak p$-bounded nor $\theta(a,y)$ is $\frak q$-bounded formula. If in addition $\theta(a,y)\notin\mathfrak q_{\strok Aa}$ then $\theta(a,\Mon)\cap q(\Mon)$ is bounded from above and unbounded from below in $q(\Mon)$, and the similar holds if $\theta(x,b)\notin\frak p_{\strok Ab}$. We say that a formula $\theta(a,y)$ is bounded from above (below) within $q(\Mon)$ if $\theta (a,\Mon)$ meets $q(\Mon)$, and there exists $b\in q(\Mon)$ such that \ $\theta(a,\Mon)\cap q(\Mon)<b$\ \  (\ $b<\theta(a,\Mon)\cap q(\Mon)$\ ).

\begin{lem}\label{Lunbdd_theta*}
Suppose that $a\models p$, $b\models q$ and  $\theta(x,y)\in\tp(a,b/A)$ are such that $\theta(a,y)\notin\mathfrak q_{\strok Aa}$. Then there is a formula $\theta^d(x,y)\in\tp(a,b/A)$  such that: \  $\theta^d(a,\Mon)$ is    downwards closed (with respect to  $\leq_\frak q$) subset of $\phi_\frak q(\Mon)$ and $\theta^d(a,\Mon)\cap q(\Mon)$ is the downwards closure of  $\theta(a,\Mon)\cap q(\Mon)$ in $q(\Mon)$.   
\end{lem}

\setcounter{equation}{0}
\begin{proof}Let $b_1\models q$ be an upper bound for $\theta(a,\Mon)\cap q(\Mon)$. Then \ $\{\theta(a,y)\}\cup q(y)\vdash y<_\frak q b_1$  \ and, by compactness, there is $\varphi(y)\in q(y)$ such that $\models (\theta(a,y)\wedge\varphi(y))\Rightarrow y<_\frak q b_1$.  Define:
$$\theta^d(a,y):=\phi_\frak q(y)\wedge(\exists z)(\varphi(z)\wedge \theta(a,z)\wedge y\leq_\frak q z) \ . $$
Clearly, $\theta^d(a,\Mon)$ is   a downwards closed   subset of $\phi_\frak q(\Mon)$, so it remains to show that $\theta^d(a,\Mon)\cap q(\Mon)$ is the downwards closure of  $\theta(a,\Mon)\cap q(\Mon)$. For, suppose that $b'\in\theta^d(a,\Mon)\cap q(\Mon)$ and we will prove that $b'$ is $\leq_\frak q$-smaller than some element of $\theta(a,\Mon)\cap q(\Mon)$. Let $c$ witness the existential quantifier in $\models\theta^d(a,b')$. Then    $c\in \varphi(\Mon)\cap\theta(a,\Mon)$ implies $c<_\frak q b_1$ and,   since  $b'\leq_\frak q c$ holds, we have $b'\leq_\frak q c<_\frak q b_1$.   $q(\Mon)$ is $\leq_\frak q$-convex, so $c\models q$. Thus $b'\leq_\frak q c$ and $c\in\theta(a,\Mon)\cap q(\Mon)$, as desired. The proof of the lemma is complete.
\end{proof}

\begin{lem}\label{L_bdd_rel_def}
Suppose that $a\models p$, $b\models q$ and  $\theta(x,y)\in\tp(a,b/A)$ are such that $\theta(a,y)\notin\mathfrak q_{\strok Aa}$. Then:  
\begin{enumerate}[(i)]
\item \ \ $\theta^d(a,\Mon)\cap q(\Mon)=D_{\mathfrak q}(a)$. 

\item \ \ $D_{\mathfrak q}(a)$ is relatively definable over $Aa$ within $q(\Mon)$ and $D_{\mathfrak p}(b)$ is relatively definable over $Ab$ within $p(\Mon)$. 
\end{enumerate}   
\end{lem} 

\begin{proof}  (i) \ Let  $\theta^d(a,y)$ be given by Lemma \ref{Lunbdd_theta*}. Then $\theta^d(a,\Mon)\cap q(\Mon)$ is the downwards closure of $\theta(a,\Mon)\cap q(\Mon)$, and it is clear that $\theta^d(a,y)$ is bounded from above within $q(\Mon)$ (by $b_1$ from the proof of Lemma \ref{Lunbdd_theta*}). Hence $\theta^d(a,y)\notin\frak q_{\strok Aa}$, and therefore $\theta^d(a,\Mon)\cap q(\Mon)\subseteq D_{\mathfrak q}(a)$.

Suppose that (i) fails to be true. Then there is $b'\in D_{\mathfrak q}(a)\smallsetminus \theta^d(a,\Mon)$. Let $\varphi(a,y)\in \tp(b'/Aa)$ witness $b'\in D_{\mathfrak q}(a)$. Define  $\psi(x,y):= \varphi(x,y)\wedge \neg  \theta^d(x,y)$. To reach a contradiction it suffices to show that $\psi(a,y)\in\tp(b'/Aa)$ is $\frak q$-bounded.   $\psi(a,y)$   is  consistent with $q(y)$ because $b'$ satisfies it; $\psi(a,y)$ is bounded from above within $q(\Mon)$ because $\varphi(a,y)$ is so. It remains to show that $\psi(a,y)$  is bounded from below in $q(\Mon)$. For, it suffices to note that $\theta^d(a,y)$, being downwards closed in $q(\Mon)$, contains all $b_0$ for which $\eps_\frak q(b_0)<_\frak q b$ (because $b$ satisfies $\theta^d(a,y)$). Hence $\neg \theta^d(a,y)$ is bounded from  below by any such $b_0$, and so is $\psi(a,y)$. Therefore $\psi(a,y)$  is $\frak q$-bounded. A contradiction.   

\smallskip (ii)  Follows immediately from part (i).
\end{proof}

\begin{lem}\label{L_bdd_scl=scl}  
The following conditions are all equivalent for $a,a'$ realizing $p$: 
\begin{center}
  (1) \ $\eps_\frak p(a)=\eps_\frak p(a')$; \ \ (2) \ 
  $D_{\mathfrak q}(a)=D_{\mathfrak q}(a')$; \ \ (3)  \ 
  $\pi_\frak q(D_{\mathfrak q}(a))= \pi_\frak q(D_{\mathfrak q}(a'))$.
\end{center}
 \end{lem}
\begin{proof}
To prove (1)$\Rightarrow$(2), assume  that $\eps_\frak p(a)=\eps_\frak p(a')$. Let $\theta(a,y)$ relatively define $D_\frak q(a)$ within $q(\Mon)$ and let $a_0, a_1$ be such that $\eps_\frak p(a_0)<_\frak p\eps_\frak p(a)<_\frak p\eps_\frak p(a_1)$\,. By Lemma \ref{Lscl_rel_definable}, there is a formula $\sigma(x,x')\in\tp(a,a'/A)$ such that $\sigma(a,x')\vdash x'\in\eps_\frak p(a)$. We {\em claim} that $\psi(a,y):=\exists x'(\sigma(a,x')\wedge\theta(x',y))$ is bounded from above in $q(\Mon)$. Actually, $\psi(a,y)$ relatively defines $\bigcup_{ a''\in\sigma(a,\Mon)} \theta(a'',\Mon)\cap q(\Mon)$ and we will prove $$\bigcup_{ a''\in\sigma(a,\Mon)} \theta(a'',\Mon)\cap q(\Mon)\subseteq D_\frak q (a_0)\cup D_\frak q(a_1)\,.$$  
So let $a''$ satisfy  $\sigma(a,x')$. Then, by our choice of $\sigma$,  $a''\in\eps_\frak p(a)$. Hence $\theta(a'',\Mon)$ relatively defines  $D_\frak q(a'')$, i.e. $\theta(a'',\Mon)\cap q(\Mon)=D_\frak q(a'')$, and  $\eps_\frak p(a_0)<_\frak p a''<_\frak p\eps_\frak p(a_1)$ holds. By Lemma \ref{LDworsas}, $D_\frak q(a'')$ is strictly included between $ D_\frak q (a_0)$ and $D_\frak q(a_1)$ (in some order),   so $D_\frak q(a'')\subseteq  D_\frak q (a_0)\cup D_\frak q(a_1)$.  This proves the claim. Now $\psi(a,y)$ is bounded from above within $q(\Mon)$, so $\psi(a,\Mon)\cap q(\Mon)\subseteq D_\frak q(a)$. Since $a'$ can witness the existential quantifier in the definition of $\psi$, $\psi(a,\Mon)\supseteq\theta(a',\Mon)$ holds, and we derive $D_\frak q(a')\subseteq D_\frak q(a)$.    
The other inclusion is proved similarly, so $D_\frak q(a')= D_\frak q(a)$.

\smallskip (2)$\Rightarrow$(3) is trivial.  (3)$\Rightarrow$(2) follows from the fact that $D_\frak q(-)$ is $\eps_\frak q$-closed (Lemma \ref{LDscl_closed}). It remains to prove  $\neg$(1)$\Rightarrow\neg$(2). So assume that  $\eps_\frak p(a)\neq\eps_\frak p(a')$. Then   the set $\{a,a'\}$ can be arranged into a Morley sequence in $\mathfrak p$ over $A$. But then, by Lemma \ref{LDworsas}, one of $D_{\mathfrak q}(a)$ and $D_{\mathfrak q}(a')$ is properly contained in the other; in particular, $D_{\mathfrak q}(a)\neq D_{\mathfrak q}(a')$. The proof of the lemma is complete. 
\end{proof}

The lemma implies that  $\eps_\frak p(a)\mapsto D_\frak q(a)$ properly defines a mapping from $\Lin_A(\frak p)$ into the power set of  $\Lin_A(\frak q)$. It will turn out that the so defined mapping induces  an isomorphism (or anti-isomorphism) of Dedekind completions of $\Lin_A(\frak p)$  and $\Lin_A(\frak q)$. To prove that, we need a few more lemmas.

In the proof of the next lemma we will essentially use the assumed strong regularity. 

\begin{lem}\label{L_bdd_vdashq}  
Suppose that $a\models p$ and that $\theta(a,y)$ relatively defines $D_\frak q(a)$.
\begin{enumerate}[(i)]
\item If  $(a,a')$ is a Morley sequence in $\frak p$ over $A$  \ then \ $\phi_\frak q(y)\wedge\neg(\theta(a,y)\Leftrightarrow\theta(a',y))  \vdash q(y)$.
\item $\varphi(x,x'):=\exists z(\neg(\theta(x,z)\Leftrightarrow\theta(x',z))\wedge\phi_\frak q(z) )\,$   relatively defines $\eps_\frak p(x)\neq\eps_\frak p(x')$ within $p(\Mon)^2$.
\item  $\mathfrak p$ and $\mathfrak q$ are simple over $A$. 
\end{enumerate}
\end{lem}
\begin{proof}(i) 
Assume \  $\models\phi_\frak q(b)\wedge\neg(\theta(a,b)\Leftrightarrow\theta(a',b))$ \ and let $r=\tp(b/A)$. Then  $\models  \neg(\theta(a,b)\Leftrightarrow\theta(a',b))$ implies $r\nwor p$. $r\nwor p$ and $p\nwor q$, by Proposition \ref{prop_nwor_transitivity}, imply $r\nwor q$. Since $\phi_\frak q(x)$ belongs to $r$  and  witnesses  strong regularity of $\frak q$, we deduce $r=q$. Therefore, any element satisfying the formula realizes $q$. This completes the proof. 
 
\smallskip (ii)  
For $a,a'$ realizing $p$ the following conditions are all equivalent:
\begin{center}
 (1) \   $\models \varphi(a,a')$; \ \ \ 
 (2) \   $D_\frak q(a)\neq D_\frak q(a')$; \ \ \ 
 (3) \ $\eps_\frak p(a)\neq \eps_\frak p(a')$.
\end{center}
(1) means that some element of $\phi_\frak q(\Mon)$ is in the difference of $\theta(a,\Mon)$ and $\theta(a',\Mon)$; that element, by part (i), realizes $q$  and is in the difference of $D_\frak q(a)$ and $D_\frak q(a')$. The converse is similar, so (1) and (2) are equivalent. The equivalence of (2) and (3) is proven in Lemma \ref{L_bdd_scl=scl}. 
The equivalence of (1) and (3) implies that $\neg \varphi(x,x')$ relatively defines $\eps_\frak p(x)=\eps_\frak p(x')$ within $p(\Mon)^2$. This proves (ii) and then (iii) follows.
\end{proof}

By Proposition \ref{P_conv_simple_invariants} we immediately obtain:

\begin{cor}  $\mathbb L_A(\frak p)$ and $\mathbb L_A(\frak q)$ are dense linear orders.
\end{cor}

Let $\mathbb L=(L,<)$ be a dense linear order  with or without end points. By a  Dedekind completion of $\mathbb L$ we mean    $\mathcal D(\mathbb L)=(\mathcal D(L),\subsetneq)$, where $\mathcal D(L)$ consists of $\emptyset$ and  the set of all initial segments of $L$ which do not have (contain) maximum. $\mathcal D(L)$ is complete, $\emptyset$ is its minimum and its maximum is either $L$ (if $L$ does not have maximum) or $\{x\mid x<c\}$ (if $c$ is the maximum).  
There is a natural embedding of $L$ into $\mathcal D(L)$: the minimum  (if it exists) is mapped to $\emptyset$,    and any other $a\in L$ is mapped to $\{x\in L\mid x<a\}$. In this way $L$ is identified with a dense subset of $\mathcal D(L)$. 

\begin{lem}\label{Lun_D_no_maximum} 
 $\pi_\frak q(D_\frak q(a))\in\mathcal D(\Lin_A(\frak q))$ \ for all $a\models \frak p$. 
\end{lem}
\begin{proof}
We already know that $D_\frak q(a)$ is downwards closed, so  $\pi_\frak q(D_\frak q(a))$ is an initial segment of $\Lin_A(\frak p)$. It remains to show that $\pi_\frak q(D_\frak q(a))$ does not have maximum. Otherwise,  the formula $\psi(a,y)$ 
saying that `no element of $D_\frak q(a)$ is bigger than $\eps_\frak q(y)$' would be satisfied in $D_\frak q(a)$ exclusively by the elements of the maximum and hence $\psi(a,y)$ would be $\frak q$-bounded. A contradiction. 
\end{proof}

 By Lemmas \ref{L_bdd_scl=scl} and  \ref{Lun_D_no_maximum} the following definition is proper:   
 \begin{center}
 $F(\eps_\frak p(a)):= \pi_\frak q(D_\frak q(a))$ \ \ defines  \ \ 
  $F:\Lin_A(\frak p)\longrightarrow  \mathcal D(\Lin_A(\frak q))$\,; 
 \end{center}
 
\begin{lem}\label{LdenseF} 
\begin{enumerate}[(i)]
\item  $F:\Lin_A(\frak p)\longrightarrow  \mathcal D(\Lin_A(\frak q))$ \ is   strictly monotone. More precisely
\begin{itemize}
\item   $F$ is strictly increasing \ iff  \  $(\frak p\otimes\frak q)_{\strok A}\neq(\frak q\otimes\frak p)_{\strok A}$;
\item   $F$ is strictly decreasing \ iff \ $(\frak p\otimes\frak q)_{\strok A}=(\frak q\otimes\frak p)_{\strok A}$.
\end{itemize}
\item $F[\Lin_A(\frak p)]$ is a dense subset of $\mathcal D(\Lin_A(\frak q))$.
\end{enumerate}
\end{lem} 
\begin{proof}(i) follows immediately from  Lemma \ref{LDworsas} so we prove only (ii).
Assume that $I\subsetneq J$ are elements of $\mathcal D(\Lin_A(\frak q))$. Then, since $J$ has no maximum, there are $b,b'\models q$ such that $I<_\frak q\eps_\frak q(b)<_\frak q\eps_\frak q(b')$ and $\eps_\frak q(b'),\eps_\frak q(b)\in J$. By Lemma \ref{LDworsas} we can pick $a\models p$ such that: either $b\in D_\frak q(a)$ and $b'\in I_\frak q(a)$ (in the non-commutative case), or $b'\in D_\frak q(a)$ and $b\in I_\frak q(a)$ (in the commutative case).   Then $I<_\frak q \pi_\frak q(D_\frak q(a))<_\frak q J$ so $F[\Lin_A(\frak p)]$ is  dense. 
\end{proof}

The following describes all relevant combinatorial properties in our situation:
\begin{enumerate}
\item $\mathbb L_A(\frak p)$ and $\mathbb L_A(\frak q)$ are dense linear orders;  
\item $F:\Lin_A(\frak p)\longrightarrow  \mathcal D(\Lin_A(\frak q))$ \ is   strictly monotone; 
\item $F[\Lin_A(\frak p)]$ is a dense subset of $\mathcal D(\Lin_A(\frak q))$.
\end{enumerate}
We leave to the reader   to verify   a purely combinatorial fact that 
whenever two  linear orders and a mapping  $F$  satisfy the three conditions, then $F$ can be naturally extended to the Dedekind completion of the domain. Moreover, if $F$ is increasing then
\begin{center}
  $\mathcal D(F):\mathcal D(\Lin_A(\frak p))\longrightarrow  \mathcal D(\Lin_A(\frak q))$ \ \ defined by \ \ $\displaystyle\mathcal D(F)(I):= \bigcup_{\eps_\frak p(x)\in I}F(\eps_\frak p(x))$
\end{center}is an isomorphism of $\mathcal D(\mathbb L_A(\frak p))$ and   $\mathcal D(\mathbb L_A(\frak q))$; if $F$ is decreasing then 
\begin{center}
   $\mathcal D^*(F):\mathcal D(\Lin_A(\frak p))\longrightarrow  \mathcal D(\Lin_A(\frak q))$ \ \ defined by \ \   $\displaystyle\mathcal D^*(F)(I):= \bigcap_{\eps_\frak p(x)\in I}F(\eps_\frak p(x))$.
 \end{center} 
is an anti-isomorphism of  $\mathcal D(\mathbb L_A(\frak p))$ and   $\mathcal D(\mathbb L_A(\frak q))$.  Therefore:

\begin{prop}\label{prop_unbounded_iso_monster}(i) If $(\frak p\otimes\frak q)_{\strok A}\neq(\frak q\otimes\frak p)_{\strok A}$   then   $\mathcal D(\mathbb L_A(\frak p))$ and $\mathcal D(\mathbb L_A(\frak q))$ are isomorphic via $\mathcal D(F)$. 

\smallskip (ii)  If $(\frak p\otimes\frak q)_{\strok A}=(\frak q\otimes\frak p)_{\strok A}$   then   $\mathcal D(\mathbb L_A(\frak p))$ and   $\mathcal D(\mathbb L_A(\frak q))$ are anti-isomorphic via $\mathcal D^*(F)$. 
\end{prop}
 
Now we turn to the local case: let $M\supseteq A$ be a small model. For $a\in p(M)$ define:
\begin{center}
 $\eps_\frak p^M(a)=\eps_\frak p(a)\cap p(M)$ \ \ and  \ \   \ $D_\frak q^M(a)=D_\frak q(a)\cap q(M)$.
\end{center}
 
\begin{lem}\label{Lunbdd_local}  Let  \ $\varphi(x,x'):=\exists z(\neg(\theta(x,z)\Leftrightarrow\theta(x',z))\wedge\phi_\frak q(z) )\,$, where $a\in p(M)$ and $\theta(a,y)$ relatively defines $D_\frak q(a)$. Assume  that  \
$|\Lin_{\frak p}(M)|\geq 2$.
\begin{enumerate}[(i)]
\item For $a,a'\in p(M)$ the following conditions are all equivalent:
\begin{center}
\begin{tabular}{lll}
 (1) \   $\models \varphi(a,a')$; &
 (2) \   $D_\frak q(a)\neq D_\frak q(a')$; &
 (3) \   $\eps_\frak p(a)\neq \eps_\frak p(a')$;\\
 (4) \   $D_\frak q^M(a)\neq D_\frak q^M(a')$; &
 (5) \   $\eps_\frak p^M(a)\neq \eps_\frak p^M(a')$; &
 (6) \   $\pi_\frak q(D_\frak q^M(a))\neq \pi_\frak q(D_\frak q^M(a')$).
 \end{tabular}
\end{center}

\item If $a\in p(M)$ and $\eps_\frak p^M(a)$ is neither minimal nor maximal in $\Lin_{\frak p}(M)$ then $D_\frak q^M(a)\neq\emptyset$ and $\pi_q(D_\frak q^M(a))\in \mathcal D(\Lin_\frak q(M))$.
\end{enumerate}
\end{lem} 

\begin{proof}
(i) Let $a,a'\in p(M)$. Equivalences (1)$\Leftrightarrow$(2)$\Leftrightarrow$(3) follow from   Lemma \ref{L_bdd_vdashq}.  (5)$\Rightarrow$(3),  (4)$\Rightarrow$(2) and (6)$\Rightarrow$(4) are obvious.  (3)$\Rightarrow$(5) is easy: if (3) holds then $\eps_\frak p(a)$ and $\eps_\frak p(a')$ are disjoint; since $\eps_\frak p^M(a)$ and $\eps_\frak p^M(a')$ are non-empty  we conclude that they are disjoint, so (5) holds. 

Finally, it suffices to prove (1)$\Rightarrow$(6). Assume $\models\varphi(a,a')$, i.e. $\models \exists z(\neg(\theta(a,z)\Leftrightarrow\theta(a',z))\wedge\phi_\frak q(z))$. Choose $b\in M$ witnessing the existential quantifier in this formula. By Lemma \ref{L_bdd_vdashq} $b$ realizes $q$  so $b$ witnesses $D_\frak q(a)\neq D_\frak q(a')$. Since $D_\frak q(-)$ is $\eps_\frak q$-closed (Lemma \ref{LDscl_closed}),   $\eps_\frak q(b)$ witnesses $\pi_\frak q(D_\frak q(a))\neq \pi_\frak q(D_\frak q(a'))$. Finally,   $\pi_\frak q(D_\frak q^M(a))= \pi_\frak q(D_\frak q(a))\cap \Lin_\frak q(M)$ and $\eps_\frak q^M(b)\in\Lin_\frak q(M)$ (because $b\in q(M)$), imply that $\eps_\frak q^M(b)$ witnesses $\pi_\frak q(D_\frak q^M(a))\neq \pi_\frak q(D_\frak q^M(a')$.

(ii) Let $a\in p(M)$ be such that $\eps_\frak p^M(a)$ is neither minimal nor maximal in $\Lin_\frak p(M)$. Choose $a',a''\in p(M)$ such that $\eps_\frak p(a')<_\frak p\eps_\frak p(a)<_\frak p\eps_\frak p(a'')$. Then, by (i), $D_\frak q^M(a)\neq D_\frak q^M(a')$ and $D_\frak q^M(a)\neq D_\frak q^M(a'')$.
By Lemma \ref{LDworsas}  $D_\frak q(a)$ is strictly contained between $D_\frak q(a')$ and $D_\frak q(a'')$, so $D_\frak q^M(a)$ is contained between $D_\frak q^M(a')$ and $D_\frak q^M(a'')$. In particular, $D_\frak q^M(a)\neq \emptyset$. Thus 
$\pi_\frak q(D_\frak q^M(a))$ is   a proper  initial segment of $\Lin_\frak q(M)$. To complete the proof of the lemma it remains to show that  it has no maximum.

Suppose, on the contrary, that $\eps_\frak q(c)$ is the maximum and work for a contradiction. Let $a_0\in\{a',a''\}$ be such that   $D_\frak q^M(a_0)\subsetneq D_\frak q^M(a)$. Then $D_\frak q^M(a)\smallsetminus D_\frak q^M(a_0)$ is definable by $\theta(a,t)\wedge \neg\theta(a_0,t)\wedge\phi_\frak q(t)$. Let $\psi(a,a_0,c)$ be the formula saying that \begin{center}
`for all  $t\in D_\frak q^M(a)\smallsetminus D_\frak q^M(a_0)$ if  $t\notin\eps_\frak q(c)$ then  $t<_\frak q c$'\,. 
\end{center}  
Then $\models\psi(a,a_0,c)$ holds in $M$ and so it holds in $\Mon$, too. But then $\eps_\frak q(c)$ would be the maximum of $D_\frak q(a)$. A contradiction.
 \end{proof}

By  the lemma, whenever $|\Lin_\frak p(M)|\geq 2$, we can  localize $F$:
\begin{center}
$F_M(\eps_\frak p^M(a))=\pi_q(D_\frak q^M(a))$  \ defines \ $F_M:\Lin_\frak p(M)\longrightarrow  \mathcal D(\Lin_\frak q(M))$
\end{center}
Next we prove that the three conditions which guarantee extensibility of $F_M$ to an  isomorphism (or an anti-isomorphism) of $\mathcal D(\Lin_\frak p(M))$ and $\mathcal D(\Lin_\frak q(M))$ are satisfied.

\begin{lem} Assume that $|\Lin_{A,\frak p}(M)|\geq 2$.  
\begin{enumerate}[(i)]
\item $\mathbb L_\frak p(M)$ and $\mathbb L_\frak q(M)$ are dense linear orders;  
\item $F_M:\Lin_\frak p(M)\longrightarrow  \mathcal D(\Lin_\frak q(M))$ \ is   strictly monotone. Moreover:
	\begin{enumerate}
	\item If $(\frak p\otimes\frak q)_{\strok A}\neq(\frak q\otimes\frak p)_{\strok A}$ then $F_M$ is increasing;
	\item If $(\frak p\otimes\frak q)_{\strok A}=(\frak q\otimes\frak p)_{\strok A}$ then $F_M$ is decreasing;
	\end{enumerate}
\item $F_M[\Lin_\frak p(M)]$ is a dense subset of $\mathcal D(\Lin_\frak q(M))$.
\end{enumerate}
\end{lem}

\begin{proof}
(i) $\mathbb L_\frak p(M)$ is a dense linear order by Proposition \ref{P_conv_simple_invariants}, so   $\Lin_\frak p(M)$ is infinite. By part (ii) of the previous lemma $F_M(a)$ does not have maximum; in particular $\pi_\frak q(D_\frak q^M(a))$ is infinite. Therefore $\Lin_\frak q(M)$ is infinite and, by Proposition \ref{P_conv_simple_invariants}, $\mathbb L_\frak q(M)$ is also a dense linear order.

\smallskip(ii) The equivalence of (3) and (6) in part (i) of the previous lemma implies that  $F_M$ is injective.  The conclusion follows from Lemma \ref{LDworsas}.

\smallskip(iii) Assume that $I\subsetneq J$ are elements of $\mathcal D(\Lin_\frak q(M))$. It suffices to find $a$ realizing $p$ such that $I\subsetneq D_\frak q^M(a)\subsetneq J$. Since $J$ has no maximum  there are $b,b'\in q(M)$ such that $I<_\frak q\eps_\frak q^M(b)<_\frak q\eps_\frak q^M(b')$  and $\eps_\frak q^M(b),\eps_\frak q^M(b')\in J$. We will show that there exists  $a$ realizing $p$ such that $b'\in I_\frak q^M(a)$ and $b\in D_\frak q^M(a)$;  since $I,J$ and $D_\frak q^M(a)$ are all convex,   $I\subsetneq D_\frak q^M(a)\subsetneq J$ would follow.  

Let $\theta(x,y)$ be over $A$ such that $\theta(a,y)$ relatively defines $D_\frak q^M(a)$ for some (any) $a$ realizing $p$. Then  $\models\exists x\,(\neg(\theta(x,b)\Leftrightarrow\theta(x,b'))\wedge\phi_\frak p(x))$ holds because $(b,b')$ is a Morley sequence in $\frak q$ over $A$. Let $a\in M$ witness the existential quantifier. We argue as in the proof of Lemma \ref{L_bdd_vdashq}(i):    
Let $r=\tp(b/A)$. Then  $\models  \neg(\theta(a,b)\Leftrightarrow\theta(a,b'))$ implies $r\nwor q$. $r\nwor q$ and $p\nwor q$, by Proposition \ref{prop_nwor_transitivity}, imply $r\nwor p$. Since $\phi_\frak p(x)$ belongs to $r$  and  witnesses  strong regularity of $\frak p$, we deduce $r=p$. Therefore $a$ realizes $p$. This completes the proof of the lemma. \end{proof}

As an immediate corollary of the lemma we deduce.  

\begin{prop}\label{P2} Assume that $|\Lin_{A,\frak p}(M)|\geq 2$.
\begin{enumerate}[(i)]
\item If $(\frak p\otimes\frak q)_{\strok A}\neq(\frak q\otimes\frak p)_{\strok A}$ then $\mathcal D(F_M)$ is an isomorphism of $\mathcal D(\mathbb L_{A,\frak p}(M))$ and $\mathcal D(\mathbb L_{A,\frak q}(M))$.

\item If $(\frak p\otimes\frak q)_{\strok A}=(\frak q\otimes\frak p)_{\strok A}$ then $\mathcal D^*(F_M)$ is an isomorphism of $\mathcal D(\mathbb L_{A,\frak p}(M))$ and $\mathcal D(\mathbb L^*_{A,\frak q}(M))$. 
\end{enumerate}
\end{prop}

\noindent{\em Proof of Theorem \ref{T_nor_uvod}.} \ Part (i)  of the theorem, related to the bounded case, is completely proved in Proposition \ref{P1}. As for the unbounded case, in Proposition \ref{P2} we proved the claim of part (ii) for the case when $\Lin_{A,\frak p}(M)$ has at least two elements; similarly the claim holds if   $\Lin_{A,\frak q}(M)$ has at least two elements. The remaining case is when both of them have at most one element. But then  their Dedekind completions are one-element orders, and so are isomorphic (recall the convention that $\mathcal D(\emptyset)=\{\emptyset\}$). This finishes the proof of the theorem.\qed

\end{document}